\documentclass[a4paper,10pt]{article}
\usepackage{amsmath,amssymb}
\usepackage{graphicx}
\usepackage{mathrsfs} % \usepackage{rotfig}
\usepackage{mathdots}
\usepackage[utf8]{inputenc}
\usepackage{enumitem}
\usepackage{algorithm}
\usepackage{algpseudocode}
\usepackage{pgfplotstable}
\usepackage{tikz}
\usepackage{pgfplots}
\usepackage{scalerel}
\usepackage{multirow}
\usepackage{hyperref}
	\usepackage{amsthm}
\usepackage{pgfplotstable}

\usepackage{mathtools}
\usepackage{tikz}
\usepackage{pgfplots}
\usepackage{pgfplotstable}
\usepackage{geometry}
\usepackage{color}
\usepackage{booktabs}
\usepackage{framed}
\pgfplotsset{compat=1.9}

\pgfplotstableset{
	every head row/.style={before row=\toprule,after row=\midrule},
	clear infinite
}

\usepackage{amsopn}

\DeclareMathOperator{\diag}{diag}
\DeclareMathOperator{\trace}{trace}
 
 \newcommand{\rank}{\mathrm{rk}}

\newcommand{\norm}[1]{\lVert #1 \rVert}

\pgfplotsset{select coords between index/.style 2 args={
		x filter/.code={
			\ifnum\coordindex<#1\fi
			\ifnum\coordindex>#2\fi
		}
}}
\definecolor{lowrankcolor}{rgb}{.75,.75,.75}

\renewcommand{\leq}{\leqslant}
\renewcommand{\geq}{\geqslant}

\theoremstyle{definition}
\newtheorem{definition}{Definition}[section]
\theoremstyle{theorem}
\newtheorem{lemma}[definition]{Lemma}
\newtheorem{theorem}[definition]{Theorem}

	\title{Some algorithms for maximum volume and cross approximation of symmetric semidefinite matrices}

\author{Stefano Massei\footnote{Centre for Analysis, Scientific Computing and Applications (CASA), TU Eindhoven, Netherlands. E-mail: s.massei@tue.nl} }

\date{}

\begin{document}
	\maketitle
\begin{abstract}
	Various applications in numerical linear algebra and computer science are related to		
	selecting the  $r\times r$ submatrix of maximum volume contained in a given matrix $A\in\mathbb R^{n\times n}$. We propose a new greedy algorithm of cost $\mathcal O(n)$, for the case $A$ symmetric positive semidefinite (SPSD) and we discuss its extension to related optimization problems such as the maximum ratio of volumes. In the second part of the paper we prove that any SPSD matrix admits a cross approximation built on a principal submatrix whose approximation error is bounded by $(r+1)$ times the error of the best rank $r$ approximation in the nuclear norm. In the spirit of recent work by Cortinovis and Kressner we derive some deterministic algorithms, which are capable to retrieve a quasi optimal cross approximation with cost $\mathcal O(n^3)$. 
\end{abstract}
\section{Introduction}
Given  $A\in\mathbb R^{n\times n}$ and $r\in\mathbb N$, this work is mainly concerned with the selection of row and column subsets of indices  $I,J\subset \{1,\dots,n\}$ of cardinality $r$ with one of the following features:
\begin{itemize}
	\item[$(i)$] $A(I,\ J)$ is a \emph{maximum volume} submatrix that is $$\mathcal V(A(I,\ J))=\max_{|\widehat I|=|\widehat J|=r}\mathcal V(A(\widehat I,\ \widehat J)),\qquad \mathcal V(A(I,\ J)):=|\det(A(I,\ J))|,$$ 
	\item[$(ii)$] given another matrix $B\in\mathbb R^{n\times n}$, $(I,\ J)$ is a maximum point of
	$$
	\frac{\mathcal V(A(I,\ J))}{ \mathcal V( B(I, \ J))}=\max_{|\widehat I|=|\widehat J|=r}\frac{\mathcal V(A(\widehat I,\ \widehat J)}{ \mathcal V( B(\widehat I, \ \widehat J))},
	$$
	\item[$(iii)$] $A_{IJ}:=A(:, \ J)A(I,\ J)^{-1}A(I,\ :)$ is a \emph{quasi optimal cross approximation}, i.e., it verifies
	$$
	\norm{A-A_{IJ}}\leq p(r)\cdot \min_{\rank(C)=r}\norm{A-C},
	$$
	for a low-degree polynomial $p(\cdot)$ and a matrix norm $\norm{\cdot}$.
\end{itemize}

A connection between problems $(i)$ and $(iii)$ is given by a result of Goreinov and Tyrtyshnikov \cite{goreinov2001maximal}, which says that if $A(I,\ J)$ has maximum volume then the cross approximation $A_{IJ}$ satisfies the bound 
\begin{equation}\label{eq:maxvol-bound}
\norm{A-A_{IJ}}_{\max}\leq (r+1)\sigma_{r+1}(A),
\end{equation}  
with $\sigma_k(\cdot)$ indicating the $k$-th singular value and $\norm{\cdot}_{\max}$ denoting the maximum magnitude among the entries of the matrix argument. We remark that,  in general being a quasi optimal cross approximation does not imply any connection between the volume of $A(I, \ J)$ and the maximum volume. Indeed, while $(i)$ is an NP hard problem, it has been recently shown that a quasi optimal approximation with respect to the Frobenius norm always exists \cite{zamarashkin2018existence} and can be found in polynomial time \cite{cortinovis2019low}. 

\begin{paragraph}{Maximum volume.}
	Problem $(i)$ finds application in a varied range of fields that highlight how the maximum volume concept is multifaceted. 
	For instance, identifying the optimal nodes for polynomial interpolation on a given domain, the so called \emph{Fekete points}, can be recast as selecting the maximum volume submatrix of Vandermonde matrices on suitable discretization meshes \cite{sommariva2009computing}. 
	In the optimal experimental design of  linear regression models, it is of interest to select the subset of experiments, which is influenced the least by the noise in the measurements. To pursue this goal, the \emph{D-optimality} criterion suggests to look at the covariance matrix of the model and find its principal subblock of maximum volume \cite{kiefer1961optimum}. 
	Other fields where $(i)$ arises are rank revealing factorizations \cite{gu1996efficient,gu2004strong}, preconditioning \cite{arioli2015preconditioning} and tensor decompositions \cite{oseledets2010tt}.
	
	Finding a submatrix with either exact or approximate maximum volume are both NP hard problems \cite{ccivril2009selecting,summa2014largest}. Despite this downside there has been quite some effort in the development of efficient heuristic algorithms for volume maximization. 
	A central tool for our discussion is one of these methods: the \emph{Adaptive Cross Approximation (ACA)} \cite{bebendorf2000approximation,harbrecht2012low}. ACA is typically presented as a low-rank matrix approximation algorithm but it can be interpreted as a greedy method for maximizing the volume. When used for low-rank approximation, ACA is equivalent to a Gaussian elimination process with rook pivoting, and it returns an incomplete LU factorization. 
	%An interesting feature of ACA is that it requires to manipulate only $\mathcal O(rn)$ entries of $A$; this is particularly appealing in situations, like discretization of integral operators \cite{bebendorf2000approximation}, where evaluating the entries of the target matrix is an expensive operation. 
	In particular, the approximant computed by ACA is of the form in \eqref{eq:maxvol-bound} although there is no clear relation between the maximum volume submatrix and the submatrix selected by ACA. 
	On the other hand, the latter can be used as starting guess for procedures that ``locally maximize'' the volume, e.g., \cite{goreinov2010find,mikhalev2018rectangular}. These algorithms guarantee that the volume of the submatrix that they return can not be increased with a small cardinality change of either its row or column index set. See also \cite{osinsky2018rectangular} for an analysis of these techniques.
	
	In many situations the matrix $A$ is \emph{symmetric positive semidefinite (SPSD)}. For instance, this setting arises  in kernel-based interpolation \cite{fasshauer2015kernel}, low-rank approximation of covariance matrices \cite{harbrecht2012low,kressner2020certified} and discretization of operators involving convolution with a positive semidefinite kernel function \cite{bebendorf2008hierarchical}. The SPSD structure comes with a major benefit: the submatrix of maximum volume  is always attained for a principal submatrix \cite{cortinovis2020maximum}. Although this does not cure the NP hardness of the task, it reduces significantly the search space by adding the constraint $I=J$. 
	
	In Section~\ref{sec:maxvol} we propose a new efficient procedure for the local maximization of the volume over the set of principal submatrices. More specifically, our algorithm returns an $r\times r$ principal submatrix whose volume is maximal over the set of principal submatrices that can be obtained with the replacement of one of the selected indices. Implementation details and complexity analysis are discussed in Section~\ref{sec:details}. Numerical tests are reported in Section~\ref{sec:test1}.
\end{paragraph}
\begin{paragraph}{Maximum ratio of volumes.}
	To the best of our knowledge, there is no reference to problem $(ii)$ in the literature and there are no direct links with either $(i)$ or $(iii)$ when generic matrices $A,B$ are considered.  Nevertheless, we might think at the following situation: suppose that $A$ is SPSD, $B$ is banded and symmetric positive definite and that we want to compute a cross approximation of $E:=T_B^{-\top}AT_B^{-1}$ --- where $T_B$ indicates the Cholesky factor of $B$ --- without forming $E$. Since $E$ is SPSD  it would make sense to apply ACA with diagonal pivoting. However, this requires to evaluate the diagonal of $E$, which is as expensive as forming the whole matrix. Our idea is to replace the diagonal pivoting with the solution of $(ii)$ as heuristic strategy for finding a cross approximation for $E$.
	
	Indeed, the Binet-Cauchy theorem tells us that a principal minor of  $E$ satisfies
	\begin{align*}
	\det(E(J, \ J)) &= \sum_{|H|=|K|=r} \det(T_B^{-\top}(J, \ H))\det(A(H,\ K))\det(T_B^{-1}(K, \ J))\\
	&=\det(T_B^{-\top}(J, \ J))\det(A(J,\ J))\det(T_B^{-1}(J, \ J))\\
	&+\sum_{(H,K)\neq(J,J)} \det(T_B^{-\top}(J, \ H))\det(A(H,\ K))\det(T_B^{-1}(K, \ J)).
	\end{align*}
	If $B$ is banded and well conditioned, then $T_B$ is banded and the magnitude of the entries of $T_B^{-1}$ decays exponentially with the distance from the main diagonal \cite{demko1984decay}. Under these assumptions we might have
	\begin{equation}\label{eq:ratio-det}
	\det(E(J, \ J))\ \approx\ \det(T_B^{-\top}(J, \ J))\det(A(J,\ J))\det(T_B^{-1}(J, \ J))\ \approx \ \frac{\det(A(J, \ J))}{\det(B(J, \ J))}.
	\end{equation}
	Based on this argument we propose to  select $J$ via a greedy algorithm for $(ii)$ and return $E_J:=E(:,\ J) E(J, \ J) E(J, :)$ as approximation of $E$. Note that, forming the factors of $E_J$ only requires to solve $r$ linear systems  with $T_B$ and to compute $r$ matrix vector products with $A$.  
	
	In Section~\ref{sec:max_ratio} we describe how to extend the ACA based techniques for addressing $(i)$  to deal with $(ii)$. We conclude by testing the approximation property of the approach in Section~\ref{sec:test1}.
\end{paragraph}
\begin{paragraph}{Quasi optimal cross approximations.}
	In contrast to the typical robustness of ACA and its simple formulation, very little can be said a priori on the quality of the cross approximation that it returns. Even for structured cases, a priori  bounds for the approximation error contain factors that grow exponentially with $r$ \cite{higham2002accuracy,harbrecht2012low}, with the only exception of the doubly diagonally dominant case \cite{cortinovis2020maximum}. 
	
	Recently, Zamarshkin and Osinsky proved in \cite{zamarashkin2018existence} the existence of quasi optimal cross approximations with respect to the Frobenius norm by means of a probabilistic method. Derandomizing the proof of this result, Cortinovis and Kressner have shown in \cite{cortinovis2019low} how to design an algorithm that finds a quasi optimal cross approximation in polynomial time.
	
	In Section~\ref{sec:certified} we describe how to modify the technique used in \cite{zamarashkin2018existence} to prove that for an SPSD matrix $A$ there exists a quasi optimal cross approximation with respect to the nuclear norm which is built on a principal submatrix, i.e., $I=J$. This  is  of particular interest in uncertainty quantification: if $A$ is the covariance matrix of a Gaussian process, then the nuclear norm of the error bounds the Wasserstein distance with respect to another Gaussian process that can be efficiently sampled \cite{kressner2020certified}. 
	
	In Section~\ref{sec:alg-cert1}-\ref{sec:alg-cert2} we propose  two algorithms, obtained with the method of conditional expectations, which are able to retrieve  quasi optimal cross approximations of SPSD matrices in polynomial time. We conclude by discussing the algorithmic implementation and  reporting, in Section~\ref{sec:test2},  numerical experiments illustrating the performances of the methods. 
\end{paragraph}
\begin{paragraph}{Notation.}
	In this work we use Matlab-like notation for denoting the submatrices. The identity matrix of dimension $n$ is indicated with $\mathsf{Id}_n$ and we use $e_j$ to denote the $j$-th column of the identity matrix, whose dimension will be clear from the context. The symbols $\norm{\cdot}_*, \norm{\cdot}_F$ indicate the nuclear and Frobenius norm, respectively.
\end{paragraph}
\section{Maximizing the volume and the ratio of volumes}\label{sec:maxvol-submatrix}
Given $r\in\mathbb N$, an SPSD matrix $A\in\mathbb R^{n\times n}$ and a symmetric positive definite matrix $B\in\mathbb R^{n\times n}$, the ultimate goal of this section is to discuss some numerical methods for dealing with the following optimization problems:
\begin{align}\label{eq:maxvol}
\max_{\widehat J\subset \{1,\dots, n\}, \ |\widehat J|=r}\mathcal V(A(\widehat J,\ \widehat J)),\\
\max_{\widehat J\subset \{1,\dots, n\}, \ |\widehat J|=r}\frac{\mathcal V(A(\widehat J,\ \widehat J))}{\mathcal V(B(\widehat J,\ \widehat J))}.\label{eq:maxvol_ratio}
\end{align}
When $B=\mathsf{Id}_n$, \eqref{eq:maxvol_ratio} reduces to \eqref{eq:maxvol}; moreover \eqref{eq:maxvol} corresponds to the maximum volume problem because for an SPSD matrix, the maximum is attained at a principal submatrix \cite{cortinovis2020maximum}. We start by recalling a well known greedy strategy to deal with \eqref{eq:maxvol}, the so-called \emph{Adaptive Cross Approximation (ACA)} \cite{harbrecht2012low}. Then, we will see how to generalize ACA for addressing \eqref{eq:maxvol_ratio}.   
\subsection{Adaptive cross approximation}\label{sec:aca}
The selection of high volume submatrices of $A$ is intimately related with the low-rank approximation of $A$. The link is the \emph{cross approximation} \cite{bebendorf2000approximation,tyrtyshnikov2000incomplete}, which associates with a given subset of indices $J=\{j_1,\dots,j_r\}$, or equivalently with an invertible  submatrix $A(J, \ J)$, the rank $r$ matrix approximation\footnote{Cross approximation is generally associated with two subsets of indices, one for the rows and one for the columns of the submatrix. In view of the positive definiteness of $A$ we restrict to principal submatrices.}
\begin{equation}\label{eq:cross}
A_J:=A(:, \ J)A(J,\ J)^{-1}A(J,\ :).
\end{equation}
Cross approximations are attractive because to build $A_J$ only requires a partial evaluation of the entries of $A$, which is  crucial when considering large scale matrices. Moreover, since the residual matrix $R_J:=A-A_J$ is SPSD, the approximation error can be cheaply estimated as
\begin{equation} \label{eq:trace-error}
\trace(R_J)=\norm{R_J}_*\geq \norm{R_J}_F\geq\norm{R_J}_2\geq \frac{\trace(R_J)}{ n}.
\end{equation}
When $J$ is a maximum point of \eqref{eq:maxvol}, $A_J$ yields a quasi optimal approximation error with respect to the maximum norm \cite{goreinov2001maximal}. However, solving \eqref{eq:maxvol} is NP hard which paves the way to the use of heuristic approaches such as ACA.  

The ACA method selects $J$ with a process analogous to Gaussian elimination with complete pivoting. The algorithm begins by choosing $j_1=\arg\max_j A_{jj}$ and computes $R_{J_1}=A-A(:,\ j_1)A_{j_1j_1}^{-1}A(j_1,\ :)$. Then, the procedure is iterated on the residual matrices $R_{J_i}$, $i=1,\dots,r-1$ in order to retrieve $r$ indices. The elements $(R_{J_i})_{j_{i+1}j_{i+1}}$ correspond to the first $r$ pivots selected by the Gaussian elimination with complete pivoting on the matrix $A$, and we have the identity 
\begin{equation}\label{eq:pivot}
\det(A(J,\ J))=\prod_{i=0}^{r-1} (R_{J_i})_{j_{i+1}j_{i+1}},
\end{equation}
where $R_{J_0}:=A$. In particular, \eqref{eq:pivot} explains  that each step of ACA  augments the set of selected indices by following a greedy strategy with respect to the volume of the selected submatrix. The whole procedure is reported in Algorithm~\ref{alg:aca}. Note that, if one stores the vectors $u_1,\dots, u_r$, then  only the diagonal and the columns $j_1,\dots,j_r$, of  $A$, need to be evaluated. The efficient implementation of the algorithm replaces the computation of the  residual matrix at line~\ref{step:residual} with the update of the diagonal of $R_J$. Computing $R_J(:, j_k)= A(:,j_k)-U_{k-1}U_{k-1}(j_k,:)^\top$, $U_{k-1}:=[u_1,\dots,u_{k-1}]$, only requires a partial access to $A$ as well. In case the matrix $A$ is not formed explicitly and its entries are evaluated with a given handle function, Algorithm~\ref{alg:aca} requires $\mathcal O(rn)$ storage and its computational cost is $\mathcal O((r+c_A)rn)$ where $c_A$ denotes the cost of evaluating one entry of $A$.
\begin{minipage}[t]{.48\linewidth}
	\begin{algorithm}[H] 
		\small 
		\caption{ACA for \eqref{eq:maxvol}}\label{alg:aca}
		\begin{algorithmic}[1]
			\Procedure{aca}{$A,r$}
			\State Set $R_J:=A$, $J:= \emptyset$
			\For{$k:=1,2,\dots,r $} 
			\State $j_k := \arg\max_{j} (R_{J})_{jj} $  
			%\State $\mathsf{vol} \gets \mathsf{vol} \cdot (R_J)_{i_ki_k}$
			\State $J \gets J\cup\{j_k\}$ 
			\If{$k<r$}
			\State $u_k := R_J(:,\ j_k) / \sqrt{(R_J)_{j_kj_k}}$  
			\State $R_J \gets R_J - u_k u_k^\top$ \label{step:residual}
			\EndIf
			\EndFor
			\State \Return $J$
			\EndProcedure
		\end{algorithmic}
	\end{algorithm} 
\end{minipage}
\begin{minipage}[t]{.48\linewidth}
	\begin{algorithm}[H] 
		\small 
		\caption{ACA for \eqref{eq:maxvol_ratio}}\label{alg:aca_ratio}
		\begin{algorithmic}[1]
			\Procedure{aca\_ratio}{$A,B,r$}
			\State Set $R_J^{(A)}:=A$, $R_J^{(B)}:=B$, $J:= \emptyset$
			\For{$k:=1,2,\dots,r $} 
			\State $j_k := \arg\max_{j} (R_{J}^{(A)})_{jj} /  (R_{J}^{(B)})_{jj}$  
			%\State $\mathsf{vol} \gets \mathsf{vol} \cdot (R_J)_{i_ki_k}$
			\State $J \gets J\cup\{j_k\}$ 
			\If{$k<r$}
			\State $u_k^{(A)} := R_J^{(A)} (:,\ j_k) / \sqrt{(R_J^{(A)} )_{j_kj_k}}$  
			\State $R_J^{(A)}  \gets R_J^{(A)}  - u_k^{(A)}  (u_k^{(A)})^\top$ 
			\State $u_k^{(B)} := R_J^{(B)} (:,\ j_k) / \sqrt{(R_J^{(B)} )_{j_kj_k}}$  
			\State $R_J^{(B)}  \gets R_J^{(B)}  - u_k^{(B)}  (u_k^{(B)})^\top$ 
			\EndIf
			\EndFor
			\State \Return $J$
			\EndProcedure
		\end{algorithmic}
	\end{algorithm} 
\end{minipage}
\subsection{Local maximization}\label{sec:maxvol}
Let us suppose that a certain index set $J=\{j_1,\dots,j_r\}$ is given. Inspired by \cite{goreinov2010find}, we would like to know whether the volume of $A(J, \ J)$ is \emph{locally optimal}, in the sense that it cannot be increased with the replacement of just one of the indices in $J$. Practically, this requires to check that:  \begin{equation}\label{eq:det-update}\frac{\det(A(\widehat J,\ \widehat J))}{\det(A(J,\ J))}\leq 1,\qquad \forall \widehat J:\quad |J\cap\widehat J|=r-1,\quad |\widehat J|=r.\end{equation} 
For the low-rank approximation problem in the maximum norm, a locally optimal determinant is sufficient to reach a quasi optimal accuracy.
\begin{lemma}
	Let $A\in\mathbb R^{n\times n}$ be an SPSD matrix and let $J$ be an index set such that condition \eqref{eq:det-update} is verified. Then
	\[\norm{A-A_{J}}_{\max}\leq (r+1)\sigma_{r+1}(A).\]
\end{lemma} 
\begin{proof}
	When $n=r+1$ the submatrix $A(J,\ J)$ has maximum volume and we get the claim simply applying the result of Goreinov and Tyrtyshnikov (equation \eqref{eq:maxvol-bound}). For $n>r+1$, we remark that each diagonal entry of the residual matrix $(R_J)_{hh}$ is equal to the Schur complement of $A(J, \ J)$ in $A(\widetilde J, \ \widetilde J)$, for $\widetilde J= J\cup \{h\}$. In view of \eqref{eq:det-update}, $A(J,\ J)$ is the maximum volume $r\times r$ submatrix of $A(\widetilde J,\ \widetilde J)$ that implies
	\[
	(R_J)_{hh}\leq (r+1)\sigma_{r+1}(A(\widetilde J,\ \widetilde J))\leq (r+1)\sigma_{r+1}(A).
	\]
	Since $R$ is SPSD, $(r+1)\sigma_{r+1}(A)$ also bounds its max norm.
\end{proof}

In the following sections  we describe an efficient procedure to iteratively increase $\mathcal V(A(J,\ J))$ based on the evaluation of the  $r(n-r)$ ratios in \eqref{eq:det-update}. An algorithm for the analogous, yet simpler, task when the index replacement affects only the row or the column index set has been proposed in \cite{goreinov2010find}.

\subsubsection{Updating the determinant}
Let us remark that each $A(\widehat J,\ \widehat J)$ in \eqref{eq:det-update} is a rank-$2$ modification of the matrix $A(J,\ J)$. More precisely, if the index set $\widehat J$ is obtained by replacing $j_i\in J$ with $h\in\{1,\dots,n\}\setminus I$, then 
$$
A(\widehat J,\ \widehat J) = A(J,\ J) + U W U^\top$$
where
$$
U =\left[e_{i}\ \vline\ A(J, \ h) - A( J, \ j_i)\right], \qquad W = \begin{bmatrix}
A_{hh}+A_{j_ij_i}-2A_{hj_i}&\phantom{11}& 1\\
1&&0
\end{bmatrix},$$
and $e_i$ indicates the $i$-th vector of the canonical basis.
Applying the matrix determinant lemma yields
$$
\frac{\det(A(\widehat J,\ \widehat J))}{\det(A(J,\ J))}= \det(W^{-1})\det(W^{-1}+ U^\top A(J,\ J )^{-1}U),
$$
with
$$
W^{-1}=
\begin{bmatrix}
0&&1\\ 1&\phantom{11}&2A_{hj_i}-A_{hh}-A_{j_ij_i}
\end{bmatrix},\qquad \det(W^{-1})=-1.$$
By denoting with $D:= A(J,\ J)^{-1},B := A(:,\ J)D$ and with $C:=B A(J,\ :)$, we have that 
$$
U^\top A(J,\ J)^{-1}U = \begin{bmatrix}
D_{ii}& &B_{h i}-1\\
B_{hi}-1&\phantom{11}& [B(h,\ :)-B(j_i, \ :)][A(J, \ h)-A(J, \ j_i)]
\end{bmatrix}
$$
where we have used the identities
\begin{align*}
[A(h,\ J)-A(j_i,\ J)]A(J,\ J)^{-1} &= B(h,\ :)-B(j_i,\ :),\\
[B(h,\ :)-B(j_i,\ :)]e_{i} &= B_{hi}-1.
\end{align*}
Putting all pieces together we get
$$
W^{-1}+U^\top A(J,\ J)^{-1}U = \begin{bmatrix}
D_{ i i}&& B_{h i}\\
B_{h i}&\phantom{11} &C_{hh}-A_{hh}
\end{bmatrix}.
$$
Then, we might think at the following greedy scheme for increasing  the volume of a starting submatrix $A(J,\ J)$:
\begin{itemize}
	\item[1.] Compute the Cholesky decomposition $R^\top R=A(J,\ J)$, $\qquad\qquad\qquad\qquad\qquad \mathcal O(r^3)$,
	\item[2.] Retrieve the quantities $D_{ii}$ by solving $R^\top Rx = e_i$, $i=1,\dots, r,\ $ $\qquad\qquad\qquad\ \mathcal O(r^3)$,
	\item[3.] Compute $B=A(:,\ J)(R^\top R)^{-1}$, $\qquad\qquad\qquad\qquad\qquad\qquad\qquad\qquad\qquad\ \ \mathcal O((r+c_A)rn)$,
	\item[4.] Compute $C_{hh}$ $\forall h\in \{1,\dots,n\}\setminus J$, $\qquad\qquad\qquad\qquad\qquad\qquad\qquad\qquad\qquad\ \ \mathcal O(r(n-r))$,
	\item[5.] Compute $\mathcal V_{hi}:=\left|\det\left(\begin{bmatrix}D_{ ii}&& B_{h i}\\ B_{hi}&\phantom{11}& C_{hh}-A_{hh}\end{bmatrix}\right)\right|\ $ $\forall j_i\in J$, $\forall h\in\{1,\dots,n\}\setminus J$, $\ \ \mathcal O(r(n-r))$,
	\item[6.] Identify $\mathcal V_{\hat h \hat i}=\max_{h,i} \mathcal V_{hi}$. If  $\mathcal V_{\hat h \hat i}>1+\mathsf{tol}$ --- for a prescribed tolerance $\mathsf{tol}$ --- then update $J$ by replacing $j_{\hat i}$ with $\hat h$ and repeat the procedure. Otherwise stop the iteration. 
\end{itemize}
We will discuss possible improvements to this algorithm in the next section.
\subsubsection{Updating the quantities $B,C$ and $D$}\label{sec:details}
The previously sketched procedure requires, whenever the index set $J$ is updated, to recompute the quantities $B$, $C$ and $D$. Here, we explain how to leverage the old information to decrease the iteration cost. In the following, we assume that the new index $J_{\mathrm{new}}$ is obtained by replacing  $j_i\in J_{\mathrm{old}}$ with the index $h\in\{1,\dots,n\}\setminus J_{\mathrm{old}}$.

The new matrix $D$ is the inverse of a rank-2 modification of the old $D$, therefore it can be obtained with the Woodbury identity:
\begin{equation}\label{eq:Dupdate}
D_{\mathrm{new}}\gets D_{\mathrm{old}} \ -\ \underbrace{\begin{bmatrix}e_{i}^\top A(J_{\mathrm{old}},\ J_{\mathrm{old}})^{-1}\\B(h,\ :)-B(j_i,\ :) \end{bmatrix}^\top \begin{bmatrix}
	D_{hh}&& B_{hi}\\
	B_{hi}&\phantom{11}& C_{hh}-A_{hh}
	\end{bmatrix}\begin{bmatrix}e_{i}^\top A(J_{\mathrm{old}},\ J_{\mathrm{old}})^{-1}\\B(h,\ :)-B(j_i,\ :) \end{bmatrix}}_{\Delta D}.
\end{equation}
The decomposition $R_{\mathrm{new}}^\top R_{\mathrm{new}}=A(J_{\mathrm{new}}, \ J_{\mathrm{new}})$, can be computed with cost $\mathcal O(r^2)$ by rewriting $UWU^\top=\widetilde u_1\widetilde u_1^\top-\widetilde u_2\widetilde u_2^\top$, i.e., as the difference of two rank-1 SPSD matrices, and performing a rank-$1$ update and a rank-$1$ downdate of the old Cholesky factor \cite[Chapter 4, Section 3]{stewart1998matrix}. For instance, these routines are implemented in the Matlab command \texttt{cholupdate}.

The new matrix $B$ is also a low-rank correction of the old $B$, given by
\begin{equation}\label{eq:Bupdate}
B_{\mathrm{new}}\gets B_{\mathrm{old}} \ + \ \underbrace{[A(:, \ h) - A(:, \ j_i)]e_{i}^\top (R^\top_{\mathrm{new}}R_{\mathrm{new}})^{-1} + A(:,\ J_{\mathrm{old}}) \Delta D}_{\Delta B} .
\end{equation}

Performing the updates of $D$ and $B$ with \eqref{eq:Dupdate} and \eqref{eq:Bupdate}, respectively, brings down the iteration cost to $\mathcal O(r^2+(r+c_A)n)$, apart from the first iteration which remains $\mathcal O(r^3+(r+c_A)rn)$. The procedure is reported in Algorithm~\ref{alg:maxvol}. 

Since the use of the Woodbury identity is sometimes prone to numerical instabilities, e.g, when the selected submatrix is nearly singular, we may switch off the updating mechanism by setting the boolean variable $\mathsf{do\_update}$ to false at line~$3$.

\begin{algorithm} \small 
	\caption{Local maximization of the volume}\label{alg:maxvol}
	\begin{algorithmic}[1]
		\Procedure{local\_maxvol}{$A,J,\mathsf{tol}$}
		\State Set $\mathsf{vol\_ratio}:=+\infty$, $k:=1$ \State Set $\mathsf{do\_update}:=$true/false\Comment{Enable or disable the update of $B$ and $D$}
		\While{$\mathsf{vol\_ratio}>1+\mathsf{tol}$} 
		\If{$k==1$ or $\mathsf{do\_update}==$ false}
		\State $R=$ \texttt{chol}$(A(J,J))$
		\State $B=A(:,\ J)(R^\top R)^{-1}$  
		\State Compute $D_{ii}$ by solving $R^\top Rx = e_i$, $\quad \forall i=1,\dots,r$
		\Else
		\State $R\gets$ \texttt{cholupdate}$(R, U, W)$
		\State Update $D$ via \eqref{eq:Dupdate}
		\State Update $B$ via \eqref{eq:Bupdate}
		\EndIf 
		\State $C_{hh}\gets B(h,\ :) A(J, \ h)$, $\quad\forall h\in \{1,\dots, n\}\setminus J$ \label{step:updateC}
		\State $\mathcal V_{hi}= \left|\det\left(\begin{bmatrix}D_{ ii}& B_{h i}\\ B_{hi}& C_{hh}-A_{hh}\end{bmatrix}\right)\right|\ $ $\forall j_i\in J$, $\forall h\in\{1,\dots,n\}\setminus J$  
		\State $\left[\hat h, \ \hat i\right]\gets \arg\max_{h,i} \mathcal V_{h,i}$ 
		\State $\mathsf{vol\_ratio}\gets \mathcal V_{\hat h\hat i}$
		\State $J\gets J\cup \{\hat h\}\setminus\{j_{\hat i}\}$, $k \gets k+1$
		\EndWhile
		\State \Return $J$
		\EndProcedure
	\end{algorithmic}
\end{algorithm} 

%This might be worthwhile if the algorithm takes more than one iteration and $r$ is not that small. 
%In addition, we empirically observe that the larger is $r$ the larger is the number of iterations of the maxvol procedure. A possible explanation is that the larger is $r$ then the larger is the probability that \texttt{aca\_ratio} returns a suboptimal set of indices. 

Finally, we remark that updating the diagonal elements of $C$ with the relation
$$
C_{\mathrm{new}}\gets C_{\mathrm{old}} \ + \ B_{\mathrm{old}}e_{i}[A(h,\ :) - A(j_i, \ :)]+  \Delta B A( J_{\mathrm{new}}, \ :),
$$
would reduce the cost of line~\ref{step:updateC} in Algorithm~\ref{alg:maxvol} of a factor $r$. However, since this does not change the complexity of the iteration 
and  requires to store additional intermediate quantities, it is not incorporated in our implementation. 
\subsubsection{A new algorithm for the maximum volume of SPSD matrices}
Quite naturally, we propose to apply Algorithm~\ref{alg:maxvol} to the index set returned by Algorithm~\ref{alg:aca} as heuristic method for solving \eqref{eq:maxvol}.
The resulting  procedure is ensured to return a locally optimal principal submatrix of $A$ --- in the sense of Section~\ref{sec:maxvol} --- whose volume is larger or equal than the one returned by ACA. For completeness, we report the method in Algorithm~\ref{alg:mvol}. 

By denoting with $\mathsf{it}$ the number of iterations performed by \textsc{local\_maxvol}, we have that the computational cost of Algorithm~\ref{alg:mvol} is $\mathcal O((r+c_A) (r+\mathsf{it})n)$. 

We also show that it is possible to provide an upper bound for $\mathsf{it}$ that does not depend on $n$. Finding the maximum volume submatrix of an SPSD matrix is in one to one correspondence with selecting the columns of maximum volume in its Cholesky factor $T_A$ such that $A=T_A^\top T_A$ \cite[Section 2.1.1]{cortinovis2020maximum}. In particular, the greedy algorithm for column selection, i.e. the partial QR with column pivoting, executed on $T_A$ returns the same index set $J_{\mathsf{aca}}$ identified by \textsc{aca}($A$, $r$), as they are both based on greedy unit augmentations of the index set. Moreover, the volume of $T_A(:,J_{\mathsf{aca}})$  is at least $(r!)^{-1}$ times the maximum volume achievable with a subset of $r$ columns \cite[Theorem 11]{ccivril2009selecting} and is equal to the square root of $\det(A(J_{\mathsf{aca}},J_{\mathsf{aca}}))$. Then, we have
\[
\det(A(J_{\mathsf{aca}},J_{\mathsf{aca}}))\geq \frac{\det(A(J_{\mathsf{best}},J_{\mathsf{best}}))}{(r!)^2},
\]
where $A(J_{\mathsf{best}},J_{\mathsf{best}})$ denotes the maximum volume $r\times r$ submatrix. This means that when calling \textsc{local\_maxvol} in Algorithm~\ref{alg:maxvol}, the volume cannot be increased more than a factor $(r!)^2$. Since each iteration of \textsc{local\_maxvol} increases the volume of at least a factor $1+\mathsf{tol}$, this yields the following bound on its number of iterations:
\[
(1+\mathsf{tol})^{\mathsf{it}}\leq (r!)^2\quad \Longrightarrow\quad \mathsf{it}\leq 2\frac{\log(r!)}{\log(1+\mathsf{tol})}.
\] 
Finally, by means of the Stirling's approximation, we get  $\mathsf{it}\leq 2\frac{(r+1)\log(r)-r+1}{\log(1+\mathsf{tol})}=\mathcal O(r\log(r))$.
%\vspace{-.4cm}

\noindent
\begin{minipage}[t]{.48\linewidth}
	\begin{algorithm}[H] 
		\small 
		\caption{}\label{alg:mvol}
		\begin{algorithmic}[1]
			\Procedure{maxvol}{$A,r,\mathsf{tol}$}
			\State $J=$ \Call{ACA}{$A,r$}
			\State $J\gets$ \Call{local\_maxvol}{$A,J,\mathsf{tol}$}
			\EndProcedure
		\end{algorithmic}
	\end{algorithm} 
\end{minipage}
\begin{minipage}[t]{.48\linewidth}
	\begin{algorithm}[H] 
		\small 
		\caption{}\label{alg:mvol_ratio}
		\begin{algorithmic}[1]
			\Procedure{maxvol\_ratio}{$A,B,r,\mathsf{tol}$}
			\State $J=$ \Call{ACA\_ratio}{$A,B,r$}
			\State $J\gets$ \Call{local\_maxvol\_ratio}{$A,B,J,\mathsf{tol}$}
			\EndProcedure
		\end{algorithmic}
	\end{algorithm} 
\end{minipage}
\subsection{Algorithms for maximizing the ratio of volumes }\label{sec:max_ratio}
Let $J=\{j_1,\dots,j_r\}$  be the index set at the current iteration  of either Algorithm~\ref{alg:aca} or Algorithm~\ref{alg:maxvol}. The two algorithms compute the gain factor $\det(A(\widehat J,\ \widehat J))/\det(A(J,\ J))\vphantom{\widehat{E}}$ for all the modifications $\widehat J\in \mathcal J_{\mathrm{aca}}$ and $\widehat J\in \mathcal J_{\mathrm{lmvol}}$, respectively, where
$$
\mathcal J_{\mathrm{aca}}=\{\widehat J\subset \{1,\dots,n\}:\ J\subset \widehat J,\ |\widehat J| = r+1\},\ \  \mathcal J_{\mathrm{lmvol}}=\{\widehat J\subset \{1,\dots,n\}:\ |J\cap \widehat J|=r-1,\ |\widehat J| = r\}
.$$ Therefore, Algorithm~\ref{alg:aca} and Algorithm~\ref{alg:maxvol} can be adapted for the ratio of volume problem \eqref{eq:maxvol_ratio} with the following idea: run in parallel the procedure for the matrices $A,B$ and then identify the maximum ratio of gain factors $$\frac{\det(A(\widehat J,\ \widehat J))\det(B(J,\ J))}{\det(A(J,\ J))\det(B(\widehat J,\ \widehat J))\vphantom{\widehat{E^{(B)}}}}\qquad \forall \widehat J\in\mathcal J_{\mathrm{aca}}\text{ or } \forall \widehat J\in\mathcal J_{\mathrm{lmvol}}.$$
For instance, the extension of ACA to \eqref{eq:maxvol_ratio} looks for $\arg\max_j (R_J^{(A)})_{jj}/\vphantom{\widehat{E^{(B)}}}(R_J^{(B)})_{jj}$ when choosing the next pivot element; see Algorithm~\ref{alg:aca_ratio}. 
Analogously, the version of Algorithm~\ref{alg:maxvol} which deals with the ratio of volumes, identifies the pair of indices $(h,i)$ which maximizes $\mathcal V_{hi}^{(A)}/\mathcal V_{hi}^{(B)}$.  We refer to the latter with \textsc{local\_maxvol\_ratio} and --- 
due to its length --- we refrain to write its pseudocode. Finally, the extension of Algorithm~\ref{alg:mvol} to \eqref{eq:maxvol_ratio} is reported in Algorithm~\ref{alg:mvol_ratio}.

By denoting with $\mathsf{it}$ the number of iterations performed by \textsc{local\_maxvol\_ratio}, we have that the computational cost of Algorithm~\ref{alg:mvol_ratio} is $\mathcal O((r+c_A+c_B) (r+\mathsf{it})n)$, where $c_B$ indicates the cost of evaluating one entry of $B$. 
\subsection{Numerical results}\label{sec:test1}
Algorithms~\ref{alg:aca}--\ref{alg:mvol_ratio} have been implemented in Matlab version R2020a and all the numerical tests in this work have been executed on a Laptop
with the dual-core Intel Core i7-7500U 2.70 GHz CPU, 256 KB of level 2 cache, and 16 GB of RAM. The parameter $\mathsf{tol}$ used in Algorithm~\ref{alg:mvol} and Algorithm~\ref{alg:mvol_ratio} has been set to $5\cdot 10^{-2}$ for all the experiments reported in this manuscript. In the numerical tests involving the test matrix $A_3$ and Algorithm~\ref{alg:maxvol} the updating mechanism has been switched off by setting $\mathsf{do\_update}$ to false. Everywhere else, $\mathsf{do\_update}$ has been set to true.

The code is freely available at \url{https://github.com/numpi/max-vol}.
\begin{paragraph}{Test matrices}
	Let us define five  SPSD matrices $A_1,A_2,A_3,A_4,A_5\in\mathbb R^{n\times n}$ which are involved in the numerical experiments that we are going to present:
	\begin{itemize}
		\item $(A_1)_{ij}:=\mathrm{exp}(-0.3 \ |i-j| / n)$,
		\item $(A_2)_{ij}:= \min\{i,j\}$,
		\item $(A_3)_{ij}:=\frac{1}{i+j-1}$ (Hilbert matrix),
		\item $A_4:= \mathrm{trid}(1,1,1)\otimes \mathsf{Id}_6+\mathsf{Id}_{\frac n6}\otimes  \mathrm{trid}(-0.34,1.7,-0.34)$, 
		\item $A_5:=Q\diag(d)Q^\top$,\ \  $d_{i}:=\rho^{i-1}$, $\rho\in(0,1)$,  and $Q$ is the eigenvector matrix of $\mathrm{trid}(-1,2,-1)$,
	\end{itemize}
	with $\otimes$ indicating the Kronecker product. The aforementioned test matrices are representative of various singular values distributions. $A_1,A_2$ have a subexponential decay, $A_3,A_5$ have an exponential decay and $A_4$, taken from \cite{haber2016sparse}, is banded and well conditioned.  We also indicate with $T_4^\top T_4=A_4$ the Cholesky factorization of $A_4$. When running the numerical algorithms, the matrices $A_1,A_2,A_3$ and $A_4$ are provided as function handles. Instead, the matrix $A_5$ is formed explicitly. 
\end{paragraph}
\begin{paragraph}{Test 1.}
	As first experiment we run Algorithm~\ref{alg:aca} and Algorithm~\ref{alg:mvol} on $A_1,A_2,A_3$, by setting $n=1020$ and varying the size $r$ of the sought submatrix. For the matrices $A_1,A_2$ we let $r$ to range in $\{1,\dots,100\}$. When experimenting on $A_3$ we  consider $r\in\{1,\dots,20\}$ because of the small numerical rank of the Hilbert matrix. We measure the timings required by the two methods and the gain factor $|\det(A(J_{\mathrm{maxvol}}, J_{\mathrm{maxvol}}))/ \det(A(J_{\mathrm{aca}}, J_{\mathrm{aca}}))|$ which Algorithm~\ref{alg:mvol} provides with respect to Algorithm~\ref{alg:aca}. From the results reported in Figure~\ref{fig:aca-maxvol}, we see that the costs of both algorithms scale quadratically with respect to the parameter $r$. For small values of $r$ \textsc{maxvol} struggles to increase the volume of the submatrix returned by $\textsc{aca}$. This happen more often and more consistently for larger values of $r$. We mention that disabling the updates based on the Woodbury identity generally increases of about $20$\% the timings of Algorithm~\ref{alg:mvol} for this test.  
	
	\begin{figure}
		\begin{minipage}{.49\linewidth}
			\begin{tikzpicture}
			\begin{semilogyaxis}[xlabel = $r$, ylabel = Time (s),width=\linewidth,
			height=1\linewidth,legend pos = south east]
			\addplot+[mark size=1pt, mark = o] table[x index = 0, y index = 1]
			{maxvol.dat};
			\addplot+[mark size=1pt, mark = square] table[x index = 0, y index = 2]
			{maxvol.dat};		
			\addplot[domain=1:100, dashed]{1e-5*x^2};	 		
			\legend{\textsc{aca}, \textsc{maxvol},$\mathcal O(r^2)$}
			\end{semilogyaxis}
			\end{tikzpicture}
			\begin{tikzpicture}
			\begin{semilogyaxis}[xlabel = $r$, ylabel = Time (s),width=\linewidth,
			height=1\linewidth,legend pos = south east]
			\addplot+[mark size=1pt, mark = o] table[x index = 0, y index = 1]
			{maxvol_hilb.dat};
			\addplot+[mark size=1pt, mark = square] table[x index = 0, y index = 2]
			{maxvol_hilb.dat};		
			\addplot[domain=1:20, dashed]{1e-5*x^2};	 		
			\legend{\textsc{aca}, \textsc{maxvol},$\mathcal O(r^2)$}
			\end{semilogyaxis}
			\end{tikzpicture}
		\end{minipage}~\begin{minipage}{.49\linewidth}
			\begin{tikzpicture}
			\begin{semilogyaxis}[xlabel = $r$, ylabel = Time (s),width=\linewidth,
			height=1\linewidth,legend pos = south east]
			\addplot+[mark size=1pt, mark = o] table[x index = 0, y index = 4]
			{maxvol.dat};
			\addplot+[mark size=1pt, mark = square] table[x index = 0, y index = 5]
			{maxvol.dat};		
			\addplot[domain=1:100, dashed]{1e-5*x^2};	 		
			\legend{\textsc{aca}, \textsc{maxvol},$\mathcal O(r^2)$}
			\end{semilogyaxis}
			\end{tikzpicture}
			
			\qquad
			\begin{tikzpicture}
			\begin{axis}[xlabel = $r$, ylabel =Gain factor,width=\linewidth,
			height=1\linewidth,legend pos = north west]
			\addplot+[mark size=1pt] table[x index = 0, y index = 3]
			{maxvol.dat};
			\addplot+[mark size=1pt] table[x index = 0, y index = 6]
			{maxvol.dat};		
			\addplot+[mark size=1pt] table[x index = 0, y index = 3]
			{maxvol_hilb.dat};	
			\legend{$A_1$, $A_2$,$A_3$}
			\end{axis}
			\end{tikzpicture}
		\end{minipage}
		\caption{Timings of Algorithm~\ref{alg:aca} and Algorithm~\ref{alg:mvol} on the test matrices $A_1$ (top-left), $A_2$ (top-right), $A_3$ (bottom-left) and measured gain factors (bottom-right).}\label{fig:aca-maxvol}
	\end{figure}
\end{paragraph}
\begin{paragraph}{Test 2.}
	The second numerical test considers maximizing the ratio of volumes \eqref{eq:maxvol_ratio}. We keep $n=1020$ and  we run Algorithm~\ref{alg:aca_ratio} and Algorithm~\ref{alg:mvol_ratio} using $A_1,A_2,A_3$, as numerator and $A_4$ as denominator.  The time consumption as the size $r$ of the submatrix increases is reported  Figure~\ref{fig:aca-maxvol-ratio}.  Also in this case, quadratic complexity with respect to $r$ is observed for the computational cost. The gain factor $|\det(A(J_{\mathrm{maxvol\_ratio}}, J_{\mathrm{maxvol\_ratio}}))/$ $ \det(A(J_{\mathrm{aca\_ratio}}, J_{\mathrm{aca\_ratio}}))|$ is shown as well in the bottom right part of Figure~\ref{fig:aca-maxvol-ratio}.
	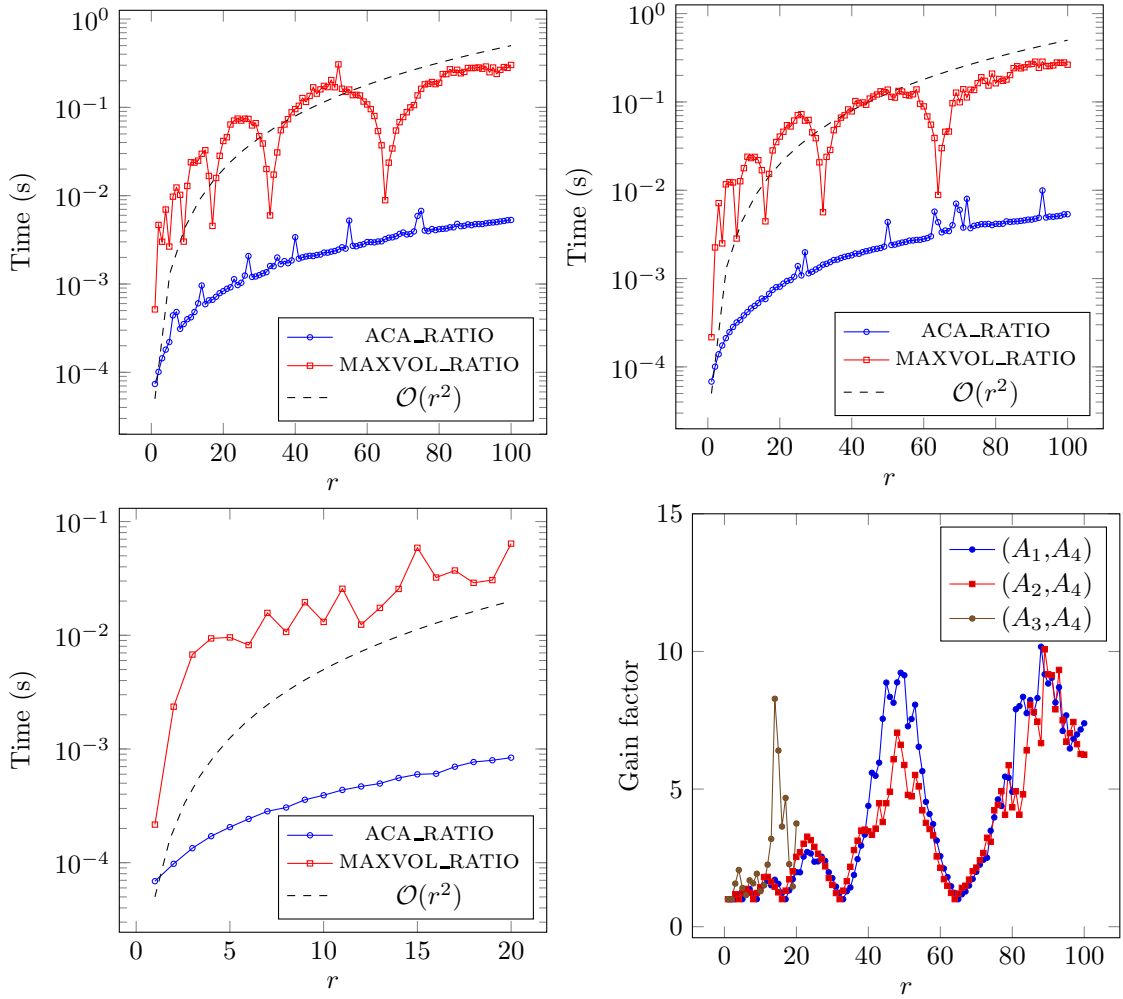
\begin{figure}
		\begin{minipage}{.49\linewidth}
			\begin{tikzpicture}
			\begin{semilogyaxis}[xlabel = $r$, ylabel = Time (s),width=\linewidth,
			height=1\linewidth,legend pos = south east]
			\addplot+[mark size=1pt, mark = o] table[x index = 0, y index = 1]
			{maxvol_ratio.dat};
			\addplot+[mark size=1pt, mark = square] table[x index = 0, y index = 2]
			{maxvol_ratio.dat};		
			\addplot[domain=1:100, dashed]{5e-5*x^2};	 		
			\legend{\textsc{aca\_ratio}, \textsc{maxvol\_ratio},$\mathcal O(r^2)$}
			\end{semilogyaxis}
			\end{tikzpicture}
			\begin{tikzpicture}
			\begin{semilogyaxis}[xlabel = $r$, ylabel = Time (s),width=\linewidth,
			height=1\linewidth,legend pos = south east]
			\addplot+[mark size=1pt, mark = o] table[x index = 0, y index = 1]
			{maxvol_ratio_hilb.dat};
			\addplot+[mark size=1pt, mark = square] table[x index = 0, y index = 2]
			{maxvol_ratio_hilb.dat};		
			\addplot[domain=1:20, dashed]{5e-5*x^2};	 		
			\legend{\textsc{aca\_ratio}, \textsc{maxvol\_ratio},$\mathcal O(r^2)$}
			\end{semilogyaxis}
			\end{tikzpicture}
		\end{minipage}~\begin{minipage}{.49\linewidth}
			\begin{tikzpicture}
			\begin{semilogyaxis}[xlabel = $r$, ylabel = Time (s),width=\linewidth,
			height=1\linewidth,legend pos = south east]
			\addplot+[mark size=1pt, mark = o] table[x index = 0, y index = 4]
			{maxvol_ratio.dat};
			\addplot+[mark size=1pt, mark = square] table[x index = 0, y index = 5]
			{maxvol_ratio.dat};		
			\addplot[domain=1:100, dashed]{5e-5*x^2};	 		
			\legend{\textsc{aca\_ratio}, \textsc{maxvol\_ratio},$\mathcal O(r^2)$}
			\end{semilogyaxis}
			\end{tikzpicture}
			
			\qquad
			\begin{tikzpicture}
			\begin{axis}[xlabel = $r$, ylabel =Gain factor,width=\linewidth,
			height=1\linewidth,legend pos = north east, ymax=15]
			\addplot+[mark size=1pt] table[x index = 0, y index = 3]
			{maxvol_ratio.dat};
			\addplot+[mark size=1pt] table[x index = 0, y index = 6]
			{maxvol_ratio.dat};		
			\addplot+[mark size=1pt] table[x index = 0, y index = 3]
			{maxvol_ratio_hilb.dat};	
			\legend{$(A_1\text{,}A_4)$, $(A_2\text{,}A_4)$, $(A_3\text{,}A_4)$}
			\end{axis}
			\end{tikzpicture}
		\end{minipage}
		\caption{Timings of Algorithm~\ref{alg:aca_ratio} and Algorithm~\ref{alg:mvol_ratio} on the test matrices $(A_1, A_4)$ (top-left), $(A_2, A_4)$ (top-right), $(A_3,A_4)$ (bottom-left) and measured gain factors (bottom-right).}\label{fig:aca-maxvol-ratio}
	\end{figure}
\end{paragraph}
\begin{paragraph}{Test 3.}
	Let us test the computational cost of \textsc{aca}, \textsc{maxvol}, \textsc{aca\_ratio} and \textsc{maxvol\_ratio} as the size of the target matrices increases. We fix $r=40$ and we let $n= 1020\cdot 2^t$, $t=0,\dots, 10$. Then, we run  \textsc{aca}, \textsc{maxvol} on $A_1$ and \textsc{maxvol}, \textsc{aca\_ratio} on the pair $(A_1, A_4)$. The timings reported in Figure~\ref{fig:aca_times} confirm that the computational time scales linearly with respect to $n$.
	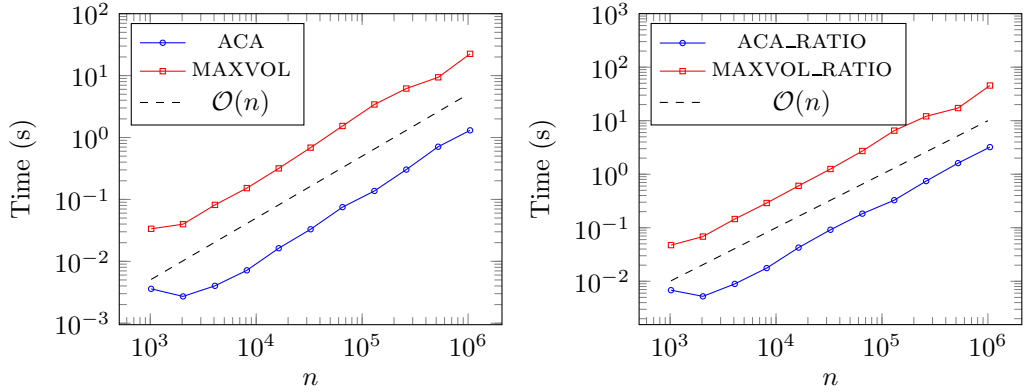
\begin{figure}
		\begin{tikzpicture}
		\begin{loglogaxis}[xlabel = $n$, ylabel = Time (s),width=.45\linewidth,legend pos = north west, ymax = 1e2]
		\addplot+[mark size=1pt, mark = o] table[x index = 0, y index = 1]
		{aca_times.dat};
		\addplot+[mark size=1pt, mark = square] table[x index = 0, y index = 2]
		{aca_times.dat};		
		\addplot[domain=1020:1e6, samples =20, dashed]{5e-6*x};	 		
		\legend{\textsc{aca}, \textsc{maxvol},$\mathcal O(n)$}
		\end{loglogaxis}
		\end{tikzpicture}~
		\begin{tikzpicture}
		\begin{loglogaxis}[xlabel = $n$, ylabel = Time (s),width=.45\linewidth,legend pos = north west, ymax = 1e3]
		\addplot+[mark size=1pt, mark = o] table[x index = 0, y index = 3]
		{aca_times.dat};
		\addplot+[mark size=1pt, mark = square] table[x index = 0, y index = 4]
		{aca_times.dat};		
		\addplot[domain=1020:1e6, samples = 20, dashed]{1e-5*x};	 		
		\legend{\textsc{aca\_ratio}, \textsc{maxvol\_ratio},$\mathcal O(n)$}
		\end{loglogaxis}
		\end{tikzpicture}
		\caption{Computational times of the algorithms as $n$ increases for $r=40$. On the left \textsc{aca} and \textsc{maxvol} have been run on the matrix $A_1$. On the right \textsc{aca\_ratio} and \textsc{maxvol\_ratio} have been run on the pair of matrices $(A_1,A_4)$. }\label{fig:aca_times}
	\end{figure}
\end{paragraph}
\begin{paragraph}{Test 4.}
	Finally, we test the quality of the cross approximations returned by \textsc{aca\_ratio} and \textsc{maxvol\_ratio}. More specifically, we compute  the approximation error $\norm{E_i - (E_i)_J }_2$, $i=1,2,3,5$, with $E_i:=(T_4^\top)^{-1}A_iT_4^{-1}$, $n =1020$ and  $J$ chosen as either $J_{\mathrm{aca\_ratio}}$ or $J_{\mathrm{maxvol\_ratio}}$. In Figure~\ref{fig:aca-maxvol-ratio-error} we compare the error curves, as $r$ increases, of the cross approximations with the ones associated with the truncated SVD, which represents the best attainable scenario. We see that the decay rate of the error of \textsc{aca\_ratio} is pretty similar to the one of the truncated SVD. \textsc{maxvol\_ratio} performs also well on the matrices which have a fast decay of the singular values, i.e., $A_3,A_5$. However, its convergence deteriorates for the matrices $A_1$ and $A_2$ and the associated error is worse than the one of \textsc{aca\_ratio}. It turns out that in these cases the approximation given in \eqref{eq:ratio-det} is less accurate and the submatrix of $(T_4^\top)^{-1}A_iT_4^{-1}$ corresponding to $J_{\mathrm{aca\_ratio}}$ has  a larger volume than the one corresponding to $J_{\mathrm{maxvol\_ratio}}$.
\end{paragraph}

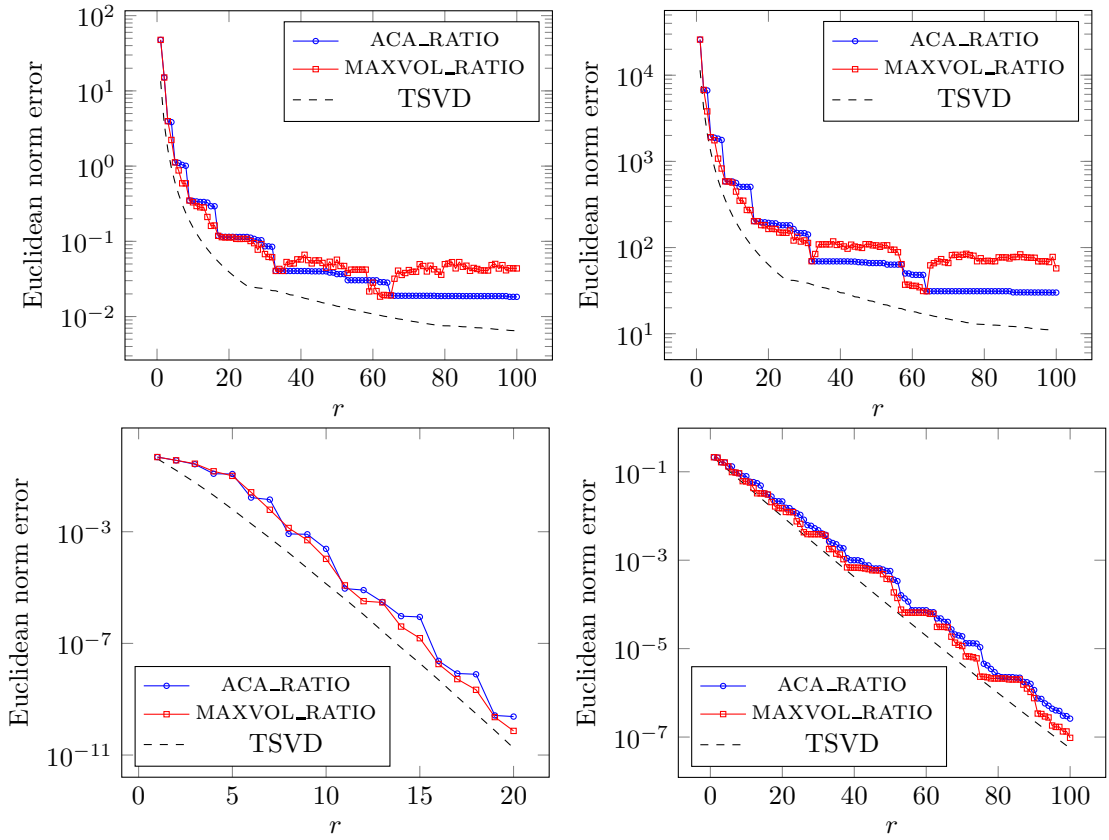
\begin{figure}
	\centering
	\begin{tikzpicture}
	\begin{semilogyaxis}[xlabel = $r$, ylabel =Euclidean norm error,width=.49\linewidth,
	legend pos = north east]
	\addplot+[mark size=1pt, mark = o] table[x index = 0, y index = 7]
	{maxvol_ratio.dat};
	\addplot+[mark size=1pt, mark = square] table[x index = 0, y index = 9]
	{maxvol_ratio.dat};		
	\addplot+[mark = none, dashed, black] table[x index = 0, y index = 11]
	{maxvol_ratio.dat}; 		
	\legend{\textsc{aca\_ratio}, \textsc{maxvol\_ratio},TSVD}
	\end{semilogyaxis}
	\end{tikzpicture}~
	\begin{tikzpicture}
	\begin{semilogyaxis}[xlabel = $r$, ylabel = Euclidean norm error,width=.49\linewidth,
	legend pos = north east]
	\addplot+[mark size=1pt, mark = o] table[x index = 0, y index = 8]
	{maxvol_ratio.dat};
	\addplot+[mark size=1pt, mark = square] table[x index = 0, y index = 10]
	{maxvol_ratio.dat};		
	\addplot+[mark = none, dashed, black] table[x index = 0, y index = 12]
	{maxvol_ratio.dat}; 		
	\legend{\textsc{aca\_ratio}, \textsc{maxvol\_ratio},TSVD}
	\end{semilogyaxis}
	\end{tikzpicture}
	\begin{tikzpicture}
	\begin{semilogyaxis}[xlabel = $r$, ylabel = Euclidean norm error,width=.49\linewidth,
	legend pos = south west]
	\addplot+[mark size=1pt, mark = o] table[x index = 0, y index = 4]
	{maxvol_ratio_hilb.dat};
	\addplot+[mark size=1pt, mark = square] table[x index = 0, y index = 5]
	{maxvol_ratio_hilb.dat};		
	\addplot+[mark = none, dashed, black] table[x index = 0, y index = 6]
	{maxvol_ratio_hilb.dat}; 		
	\legend{\textsc{aca\_ratio}, \textsc{maxvol\_ratio},TSVD}
	\end{semilogyaxis}
	\end{tikzpicture}~		\begin{tikzpicture}
	\begin{semilogyaxis}[xlabel = $r$, ylabel = Euclidean norm error,width=.49\linewidth,
	legend pos = south west]
	\addplot+[mark size=1pt, mark = o] table[x index = 0, y index = 1]
	{exp_maxvol_ratio.dat};
	\addplot+[mark size=1pt, mark = square] table[x index = 0, y index = 2]
	{exp_maxvol_ratio.dat};		
	\addplot+[mark = none, dashed, black] table[x index = 0, y index = 3]
	{exp_maxvol_ratio.dat}; 		
	\legend{\textsc{aca\_ratio}, \textsc{maxvol\_ratio},TSVD}
	\end{semilogyaxis}
	\end{tikzpicture}
	\caption{Approximation of $(T_4^\top)^{-1}A_iT_4^{-1}$ for $i=1$ (top-left), $i=2$ (top-right), $i=3$ (bottom-left), and $i=5$ (bottom-right), by means of the cross approximations associated with the outcome of Algorithm~\ref{alg:aca_ratio} and Algorithm~\ref{alg:mvol_ratio}. All plots report the  lower bound given by the error of the truncated SVD. The size of the matrices is $n=1020$. }\label{fig:aca-maxvol-ratio-error}
\end{figure}
\section{Quasi optimal cross approximation in the nuclear norm}
Adaptive cross approximation has a much lower cost than computing the truncated  SVD for the low-rank matrix approximation, although the latter provides an optimal solution, in any unitarily invariant norm. Empirically, ACA typically returns an approximant that is close, in terms of the associated approximation error, to the truncated SVD. However, it appears difficult to ensure this property theoretically, e.g., see the quite pessimistic bounds in \cite{higham2002accuracy,harbrecht2012low,cortinovis2020maximum}.
On the other hand, there are some recent results about cross approximations with quasi optimal approximation error.

Zamarashkin and Osinsky proved in \cite[Theorem 1]{zamarashkin2018existence} that, given $A\in\mathbb C^{m\times n}$ of rank $k$, $\forall r=1,\dots, k$ there exist $I=\{i_1,\dots,i_r\}\subset \{1,\dots,m\}$ and $J=\{j_1,\dots, j_r\}\subset\{1,\dots, n\}$,  such that $A(I,\ J)$ is invertible and
\begin{equation}\label{eq:quasi-frob}
\norm{A-A_{IJ}}_F\leq (r+1)\sqrt{\sum_{s\geq r+1}\sigma_s^2}, \qquad A_{IJ}:=A(:,\ J)A(I,\ J)^{-1}A(I,\ :).
\end{equation}
The authors of  \cite{zamarashkin2018existence} uses a probabilistic argument: they define the probability measure
$$
\mathbb P(A(I,\ J)) = \frac{\mathcal V(A(I,\ J))^2}{\sum\limits_{|\widehat I|=|\widehat J|=r}\mathcal V(A(\widehat I,\ \widehat J))^2}
$$
on the set  of $r\times r$ submatrices of $A$. Then, they show that $\mathbb E[\norm{A-A_{IJ}}_F]\leq (r+1)\sqrt{\sum_{s\geq r+1}\sigma_s^2}$,  which implies that there exists at least one choice of $I,J$ that verifies \eqref{eq:quasi-frob}.

Cortinovis and Kressner proposed in \cite{cortinovis2019low} a polynomial time algorithm to find $I$ and $J$ such that $A_{IJ}$ is  quasi optimal with respect to the Frobenius norm. Their approach, inspired by \cite{deshpande2010efficient}, is based on the derandomization of the result by Zamarashkin and Osinsky with the method of conditional expectations. More precisely, let $t\leq r$ and assuming to have already selected the first $t-1$ indices  $\{i_1,\dots,i_{t-1}\},\{j_1,\dots,j_{t-1}\}$  of $I$ and $J$, the pair $(i_t,j_t)$ is chosen as the one which minimizes
\begin{equation}\label{eq:cond-exp}
\mathbb E[\norm{A-A_{IJ}}_F\ | \ i_1,\dots,i_t,\ j_1,\dots,j_t].
\end{equation}
Incrementally selecting all the indices with this criteria ensures that $(I,J)$ identifies a cross approximation which verifies \eqref{eq:quasi-frob}. 
Interestingly, \eqref{eq:cond-exp} can be shown to be $(r-t+1)$ times the ratio of two consecutive coefficients in the characteristic polynomial of the symmetrized residual matrix $R_{I_tJ_t} :=(A-A_{I_tJ_t})(A-A_{I_tJ_t})^*$, with $I_t:=\{i_1,\dots,i_t\}$ and $J_t:=\{j_1,\dots,j_t\}$. The algorithm in \cite{cortinovis2019low} computes the coefficients of the characteristic polynomial of $R_{I_tJ_t}$ for all possible choices of $i_t$ and $j_t$ by updating the characteristic polynomial of $R_{I_{t-1}J_{t-1}}$; then, it chooses the pair of indices which minimizes the aforementioned ratio.

In the next section, we analyze what can be achieved with cross approximations built on principal submatrices, when $A$ is SPSD. 

\subsection{Existence result}\label{sec:certified} 
In view of \cite[Theorem 1]{cortinovis2020maximum} it is tempting to replace a symmetric choice of indices $I=J$ in \eqref{eq:quasi-frob} when $A$ is SPSD. However, such error bound it is not true in general and it is not possible to get rid of the dependency on $n$ in the multiplicative constant. For instance, consider $A=E+\epsilon\cdot \mathsf{Id}$ for a small $\epsilon> 0$ and with $E$ denoting the matrix of all ones; then, for the rank $1$ approximation of $A$, the error of the truncated SVD is $\epsilon\sqrt{n-1}$ while the one associated with any symmetric cross approximation is approximately $\epsilon(n-1)$.  The following result shows that a quasi optimal error in the nuclear norm can be  obtained by restricting the search space to principal submatrices. In view of the previous remark, this yields a sharp quasi optimal error in the Frobenius norm, with a constant increased by a factor $\sqrt{n-r}$.
\begin{theorem}\label{thm:main}
	Let $A\in\mathbb R^{n\times n}$ be SPSD of rank $k$ and $r\in\{1,\dots, k\}$. Then,
	there exists a subset of indices $J^*\subset\{1,\dots, n\}$, $|J^*|=r$ such that $A(J^*,\ J^*)$ is invertible and
	\begin{equation}\label{eq:certified}
	\norm{A-A_{J^*}}_*\leq (r+1)\cdot \sum_{s\geq r+1}\sigma_s(A),\quad \text{and}\quad \norm{A-A_{J^*}}_F\leq \sqrt{n-r}\cdot (r+1)\cdot \sqrt{\sum_{s\geq r+1}\sigma_s(A)^2}.
	\end{equation}
	%		\item[$(ii)$] If $A$ is diagonally dominant, then there exists a subset of indices $J^*\subset\{1,\dots, n\}$, $|J^*|=r$ such that $A(J^*,\ J^*)$ is invertible and
	%		\[
	%		\norm{A-A_{J^*}}_F\leq \sqrt2(r+1)\cdot \sqrt{\sum_{t\geq r+1}\sigma_t^2}.
	%		\]

\end{theorem} 
Before going into the proof of Theorem~\ref{thm:main}, let us state and prove some properties regarding the volume of principal submatrices.
\begin{lemma}\label{lem:tech1}
	Let $A\in\mathbb R^{n\times n}$ be SPSD and $J:=\{j_1,\dots,j_r\}\subset \{1,\dots, n\}$ such that $A(J,\ J)$ is invertible. Then: 
	\begin{itemize}
		\item[$(i)$] $
		\norm{A-A_J}_*=\sum\limits_{|\widehat J|=r+1,J\subset \widehat J}\frac{\mathcal V(A(\widehat J,\ \widehat J))}{\mathcal V(A(J,\ J))},%\quad\text{and}\quad \norm{A-A_J}_F^2\leq (n-r)\sum_{|\widehat J|=r+1,J\subset \widehat J}\frac{\mathcal V(A(\widehat J,\ \widehat J))^2}{\mathcal V(A(J,\ J))^2}.
		$
		\item[$(ii)$] 		$
		\sum\limits_{|J|=r}\mathcal V(A(J,\ J))= \sum\limits_{1\leq j_1<\dots <j_r\leq n}\sigma_{j_1}(A)\cdots\sigma_{j_r}(A),
		$ 
		\item[$(iii)$] for $t\in\{1,\dots,r\}$ and $J_1:=\{j_1,\dots, j_t\}\subset J$
		$$
		\sum_{j_{t+1},\dots,j_{r}}\mathcal V( A(J,\ J))=\mathcal V(A(J_1,\ J_1)) \cdot(r-t)! \cdot c_{n-r+t}(A-A_{J_1}),
		$$
		where $(-1)^{n-r+t}c_{n-r+t}(A-A_{J_1})$ indicates the coefficient which multiplies $z^{n-r+t}$ in the characteristic polynomial of $A-A_{J_1}$.
	\end{itemize}
\end{lemma}
\begin{proof}
	\fbox{$(i)$}	Let us remark that in the particular case $J=\{1,\dots, n-1\}$ we have
	$$A=\begin{bmatrix}
	A(J,\ J)& b\\
	b^\top & d	
	\end{bmatrix},\quad 
	A-A_J= \begin{bmatrix}
	0&0\\ 0&d-b^\top A(J,\ J)^{-1}b
	\end{bmatrix}$$
	and specifically:
	\begin{equation}\label{eq:res}
	\norm{A-A_J}_* =\norm{A-A_J}_F=\frac{\mathcal V(A)}{\mathcal V(A(J,\ J))},
	\end{equation}
	where the second equality has been proved in \cite[Lemma 1]{zamarashkin2018existence}.
	If $J$ is generic and $A$ is SPSD, then $A-A_J$ is SPSD and its nuclear norm is the sum of its diagonal entries which are all Schur complements of the form given in \eqref{eq:res}; this yields $(i)$.
	
	\noindent	\fbox{$(ii)$} The volume of a principal submatrix of an SPSD matrix corresponds to its determinant so that $	\sum_{|J|=r}\mathcal V(A(J,\ J))$ is equal $c_{n-r}(A)$. Since the singular values of an SPSD matrix $A$ are equal to its eigenvalues we have $c_{n-r}(A)=\sum_{1\leq j_1<\dots <j_r\leq n}\sigma_{j_1}(A)\cdots\sigma_{j_r}(A)$.  
	
	\noindent \fbox{$(iii)$} Let us denote $B:=A-A_{J_1}$ and $J_2:=J\setminus J_1$. Since $B(J_2,\ J_2)$ is the Schur complement of $A(J, \ J)$ with respect to $A(J_1,\ J_1)$ we have $\mathcal V(A(J, \ J))=\mathcal V(A(J_1,\ J_1))\mathcal V(B(J_2,\ J_2))$ so that
	$$
	\sum_{j_{t+1},\dots,j_{r}}\mathcal V( A(J,\ J))=\mathcal V( A(J_1,\ J_1))\sum_{j_{t+1},\dots,j_{r}}\mathcal V(B(J_2,\ J_2))=\mathcal V( A(J_1,\ J_1)) \cdot(r-t)! \cdot c_{n-r+t}(B),
	$$
	where the factor $(r-t)!$ accounts the repetitions in the choice of $J_2$.
	%The inequality in $(i)$ is proven similarly by noting that the square of each diagonal entry of $A-A_J$ is bigger than any of the square of the $n-r-1$ nonzero entries in the same row. 
\end{proof}

\begin{proof}[Proof of Theorem~\ref{thm:main}]
	Let us denote by $\Omega_r$  the set of $r\times r$ principal submatrices of $A$. We show that $(r+1)\cdot \sum_{t\geq r+1}\sigma_t(A)$ is larger than the expected value of the cross approximation error, with respect to the following probability distribution on $\Omega_r$:  
	$$
	\mathbb P (A(J,\ J))=\gamma\cdot \mathcal V(A(J,\ J)),\qquad \gamma:=\frac{1}{\sum_{B\in\Omega_r}\mathcal V(B)}.
	$$
	
	Indeed, we have:
	\begin{align*}
	\mathbb E [\norm{A-A_J}_*]&= \sum_{|J|=r}\mathbb P(A(J,\ J))\norm{A-A_J}_*\\
	\text{Lemma}~\ref{lem:tech1}-(i)\qquad&= \sum_{|J|=r}\sum_{|\widehat J|=r+1,J\subset \widehat J}\mathbb P(A(J,\ J))\frac{\mathcal V(A(\widehat J,\ \widehat J))}{\mathcal V(A(J,\ J))}\\
	&= \gamma\sum_{|\widehat J|=r+1}\sum_{|J|=r, J\subset \widehat J}\mathcal V(A(\widehat J,\ \widehat J))\\
	&=\gamma (r+1)\sum_{|\widehat J|=r+1}\mathcal V(A(\widehat J,\ \widehat J))\\
	\text{Lemma}~\ref{lem:tech1}-(ii)\qquad	&=\gamma(r+1)\sum_{1\leq j_1<\dots<j_{r+1}\leq n}\sigma_{j_1}(A)\cdots\sigma_{j_{r+1}}(A)\\
	&=\gamma(r+1)\sum_{1\leq j_1<\dots<j_{r}\leq n}\sigma_{j_1}(A)\cdots\sigma_{j_{r}}(A)\sum_{j_{r+1}>j_r}\sigma_{j_{r+1}}(A)\\
	&\leq\gamma(r+1)(\sigma_{r+1}(A)+\dots+\sigma_n(A))\sum_{1\leq j_1<\dots<j_{r}\leq n}\sigma_{j_1}(A)\cdots\sigma_{j_{r}}(A)\\
	\text{Lemma}~\ref{lem:tech1}-(ii)\qquad	&=(r+1)(\sigma_{r+1}(A)+\dots+\sigma_n(A)),
	\end{align*}
	where we used that once $\widehat J$ is fixed, there are $r+1$ possible choices for $J$. 
	
	Finally, we have $$\norm{A-A_J^*}_F\leq \norm{A-A_J^*}_*\leq(r+1)\sum_{s\geq r+1}\sigma_s(A)\leq \sqrt{n-r}(r+1)\sqrt{\sum_{s\geq r+1}\sigma_s(A)^2},$$ where the last inequality follows from the Cauchy–Schwarz inequality.
\end{proof}
\subsection{Derandomizing Theorem~\ref{thm:main}}\label{sec:alg-cert1} 
Following the approach in \cite{cortinovis2019low},
we obtain a deterministic algorithm for computing a cross approximation, which verifies \eqref{eq:certified}, by derandomizing Theorem~\ref{thm:main}. In order to do so, we need to determine the conditional expectation of the cross approximation error, with respect to a partial choice of the indices in $J$.
\begin{theorem} \label{thm:conditional}
	Let $A\in\mathbb R^{n\times n}$ be SPSD and $J_t:=\{j_1,\dots,j_t\}\subset \{1,\dots, n\}$ such that $A(J_t,\ J_t)$ is invertible, then
	$$
	\mathbb E (\norm{A-A_J}_*\ |\ j_1,\dots, j_t)=(r-t+1) \frac{ c_{n-r+t-1}(A-A_{J_t})}{c_{n-r+t}(A-A_{J_t})}.
	$$
\end{theorem}
\begin{proof}
	\begin{align*}
	\mathbb E (\norm{A-A_J}_*\ |\ j_1,\dots, j_t)&= \sum_{j_{t+1},\dots,j_{r}}\norm{A-A_J}_*\ \mathbb P(A(J,\ J)\ |\ j_1,\dots,j_t)\\
	&=\sum_{j_{t+1},\dots,j_{r}}\norm{A-A_J}_*\ \frac{\mathbb P(A(J,\ J))}{\mathbb P(A(J_t,\ J_t))}\\
	&=\sum\limits_{j_{t+1},\dots,j_{r}}\norm{A-A_J}_*\  \frac{\mathcal V(A(J,\ J))}{\sum\limits_{j_{t+1},\dots,j_{r+1}}\mathcal V(A(J,\ J))}\\
	\text{Lemma}~\ref{lem:tech1}-(i)\ \ \qquad&=\frac{\sum_{j_{t+1},\dots,j_{r+1}} \mathcal V(A(\{J,\ j_{r+1}\}, \{J,\ j_{r+1}\}))}{\sum_{j_{t+1},\dots,j_r}\mathcal V(A(J,\ J)}\\
	\text{Lemma}~\ref{lem:tech1}-(iii)\qquad&=(r-t+1) \frac{ c_{n-r+t-1}(A-A_{J_t})}{c_{n-r+t}(A-A_{J_t})}.
	\end{align*}
\end{proof}
Theorem~\ref{thm:conditional} suggests to design an iterative scheme that in each step computes the characteristic polynomial of $A-A_{J_t}$ for all the possible choices of the last index $j_t$ and select the one which minimizes $\frac{ c_{n-r+t-1}(A-A_{J_t})}{c_{n-r+t}(A-A_{J_t})}$. Interpreting $A-A_{J_t}$ as a rank-$1$ modification of $A-A_{J_{t-1}}$, we may look at the problem of updating the coefficients of the characteristic polynomial under a rank-$1$ change of the matrix. Since stable procedures, such as the \emph{Summation Algorithm} \cite[Algorithm 1]{rehman2011computing}, compute the characteristic polynomial from the eigenvalues, our task boils down to updating the eigenvalues of an SPSD matrix and in turn to computing the eigenvalues of a real diagonal matrix minus a rank-$1$ symmetric matrix. The latter can be transformed into a symmetric tridiagonal eigenvalue problem with a standard bulge chasing procedure \cite[Section 5]{golub1973some} and finally solved with Cuppen's divide and conquer method~\cite{cuppen1980divide}. Both tridiagonalization and Cuppen's method require $\mathcal O(n^2)$ flops.

The \emph{certified cross approximation (CCA)}  obtained from the derandomization of Theorem~\eqref{thm:main} is reported in Algorithm~\ref{alg:aca-cert}. Note that all the operations inside the inner loop have at most a quadratic cost and computing the eigendecomposition at line~\ref{step:eig} is cubic. Therefore, the asymptotic computational cost  is $\mathcal O(rn^3)$.

\begin{algorithm}
	\caption{Certified cross approximation for SPSD matrices \label{alg:aca-cert}}
	\begin{algorithmic}[1]
		\Procedure{cca}{$A,r$}
		\State{Set $R:=A,\, J:=\emptyset.$}
		\For{$t = 1\ldots,r$}
		\State Compute the eigendecomposition $R=Q\Lambda Q^\top$\label{step:eig}
		\State Compute the characteristic polynomial of $R$ via \cite[Algorithm 1]{rehman2011computing}
		\State min\_ratio $\gets\infty$
		\For{$j \in \{1\ldots,n\}\setminus J$}
		\State $u_j= R(:,\ j)/\sqrt{R(j,\ j)}$, $\quad\widetilde u_j = Q^\top u_j$
		\State Reduce $\Lambda -\widetilde u_i\widetilde u_j^\top$ to a tridiagonal matrix $T$ via bulge chasing
		\State Compute the eigenvalues of $T$ with Cuppen's method
		\State Compute the characteristic polynomial of $R-u_ju_j^\top$ via \cite[Algorithm 1]{rehman2011computing}
		\State   $\text{ratio}\gets\frac{c_{n-r+t-1}(R-u_ju_j^\top)}{\vphantom{\widehat C}c_{n-r+t}(R-u_ju_j^\top)}$
		\If {ratio $<$ min\_ratio}
		\State min\_ratio $\gets $ ratio,  $j^*\gets j$
		\EndIf
		\EndFor
		\State $R\gets R-u_{j^*}u_{j^*}^\top$
		\State $J\gets J\cup\{j^*\}$
		\EndFor
		\EndProcedure
	\end{algorithmic}
\end{algorithm}
\subsection{Updating the characteristic polynomial via trace of powers}\label{sec:alg-cert2} 
Each iteration of Algorithm~\ref{alg:aca-cert} requires to update the eigendecomposition of the residual matrix, resulting in a computational cost $\mathcal O(rn^3)$. Here we discuss how, in principle, to reduce the complexity to $\mathcal O(r^2n^\omega)$ where $2<\omega<3$ is the exponent of the computational complexity of the matrix-matrix multiplication. The idea is that, since we need to update only a (small) portion of the characteristic polynomial we may avoid to deal with the eigendecomposition. 

The coefficients of the characteristic polynomial of a matrix $A$ can be expressed with the so called \emph{Plemelj-Smithies formula} \cite[Theorem XII 1.108]{reedmethods}
\begin{equation}\label{eq:jacobi}
c_{n-k}(A)=\frac{(-1)^k}{k!}\det\left(\underbrace{\begin{bmatrix}\trace(A)& k-1\\
	\trace(A^2)&\trace(A)&k-2\\
	\vdots&\ddots&\ddots&\ddots\\
	\vdots&\ddots&\ddots&\ddots&1\\
	\trace(A^k)&\dots&\dots&\trace(A^2)&\trace(A)
	\end{bmatrix}}_{T_{(k)}}\right),
\end{equation}
so that
\begin{equation}\label{eq:det-ratio}
\frac{c_{n-(k+1)}(A)}{c_{n-k}(A)}=-\frac{1}{k+1}\frac{\det(T_{(k+1)})}{\det(T_{(k)})}.
\end{equation}
Equation~\eqref{eq:jacobi} says that for updating  the $(n-k)$-th coefficient of the characteristic polynomial it is sufficient to update the trace of the first $k$ powers of $A$ and to compute the determinant of a $k\times k$ matrix. Interestingly, if  $\trace(A),\dots,\trace(A^k)$ are known then the quantities $\trace(A-uu^\top),\dots,\trace((A-uu^\top)^k)$,  for a vector $u\in\mathbb R^n$, can be computed with a Krylov projection method.  More specifically, we have the following property \cite[Theorem 3.2]{beckermann2018low}:
$$
(A-uu^\top)^k- A^k\in\mathcal K_k(A, u):=\mathrm{span}(u, Au,\dots, A^{k-1}u).
$$
Let $H_k$ and $\widetilde H_k:=H_k- \norm{u}_2 e_1e_1^\top$ be the orthogonal projections of $A$ and $A-uu^\top$ on $\mathcal K_k(A, u)$, then it holds
\begin{equation}\label{eq:traces}
\trace((A-uu^\top)^j)- \trace(A^j) = \trace(\widetilde H_k^j)- \trace(H_k^j),\quad j=1,\dots,k.
\end{equation}
Hence, to update the traces of the first $k$ powers of $A$ we may perform $k$ steps of the Arnoldi method to get $\widetilde H_k,H_k$, compute the trace of their powers (via their eigenvalues)  and, finally, evaluate \eqref{eq:traces}. 

Updating the traces for a single low-rank modification  costs $\mathcal O(k\cdot\text{matvec}(A) + k^2n)$; so  a procedure that naively applies this computation for the $\mathcal O(n)$ low-rank modifications still provides a cubic iteration cost --- with respect to $n$ --- unless $\text{matvec}(A)$ has a subquadratic cost. In the case $\mathcal O(\text{matvec}(A))=\mathcal O(n^2)$, we propose to carry on the Arnoldi step simultaneously for all the $\mathcal O(n)$ low-rank modifications $u_{i}u_i^\top$. More specifically, if $u_{i(h)}$ denotes the $h$-th vector computed by the Arnoldi process for $\mathcal K_k(A,u_{i})$, then we perform all the Arnoldi steps together by computing the matrix-matrix multiplication $A\cdot [u_{1(h)}|\dots|u_{n-k+1(h)}]$. Theoretically, this yields the iteration cost $\mathcal O(kn^\omega)$.
This has also practical benefits because of the use of highly optimized BLAS 3 operations. 
The procedure for updating the trace of powers is reported in Algorithm~\ref{alg:trace-update}; the certified cross approximation method (CCA2) that relies on Algorithm~\ref{alg:trace-update} is reported in Algorithm~\ref{alg:cert-aca2}.

Unfortunately, Algorithm~\ref{alg:cert-aca2} suffers from the numerical instability of evaluating the determinant in \eqref{eq:jacobi}. More specifically, when the matrix $T_{(k)}$ becomes nearly singular the use of standard techniques provide small singular values, which are accurate only in an absolute sense. Methods that guarantee relative accuracy for singular values apply only to particular classes of matrices \cite{demmel1999computing,demmel2004accurate}; $T_{(k)}$ does not belong to any of such classes. On top of that, we often observe that the matrix  $T_{(k)}$ becomes nearly singular quite fast as $k$ increases; typically for $k$ above $10$ the computed ratio \eqref{eq:det-ratio} has no reliable digits. In the next section we propose a strategy to partially circumvent this problem.

\begin{minipage}[t]{.48\linewidth}
	\begin{algorithm}[H] 
		\small 
		\caption{CCA via trace of powers}\label{alg:cert-aca2}
		\begin{algorithmic}[1]
			\Procedure{cca2}{$A,r$}
			%\State $\mathbf \lambda\gets$ \texttt{eig($A$)}
			%\State $\mathbf t=(t_j)_{j=1,\dots, r}, \qquad t_j=\sum_{h=1}^n \lambda_h^j$
			\State Compute  $\mathbf t=(t_j)_{j=1,\dots, r+1},  t_j=\trace(A^j)$ via $r$ matrix-matrix multiplications
			\State Set $R:=A$, $J:= \emptyset$
			\For{$k:=1,2,\dots,r $} 
			\State $u_{j_h}=R(:,\ j_h)/\sqrt{R(j_h,\ j_h)}$, $\quad j_h\not\in  J$
			\State $U=[u_{j_1}|\dots|u_{j_{n-r+1}}]$
			\State $T\gets$ \Call{update\_traces}{$A,\mathbf t, U, r-k+2$}
			\State Set min\_ratio=$\infty$
			\For{$j_h\in\{1,\dots n\}\setminus J$}
			\State  $r_{j_h}\gets\frac{c_{n-(t+1)}(R_J-u_{j_h}u_{j_h}^\top)}{\vphantom{\widehat{ C_{j_h}^\top}}c_{n-t}(R_J-u_{j_h}u_{j_h}^\top)}$ via \eqref{eq:det-ratio}
			\If {$r_{j_h}<$ min\_ratio}
			\State min\_ratio $\gets r_{j_h}$, $h^*\gets h$
			\EndIf
			\EndFor
			\State $\mathbf t\gets T(h^*,\ 1:r-k+1)$
			\State $R\gets R-u_{j_{h^*}}u_{j_{h^*}}^\top$
			\State $J\gets J\cup\{j_{h^*}\}$
			\EndFor
			\State \Return $J$
			\EndProcedure
		\end{algorithmic}
	\end{algorithm} 
\end{minipage}
\begin{minipage}[t]{.48\linewidth}
	\begin{algorithm}[H] 
		\small 
		\caption{Update the trace of powers \vspace{1pt} }\label{alg:trace-update}
		\begin{algorithmic}[1]
			\Procedure{update\_traces}{$A,\mathbf t, U,k$}
			\State Set $V_j=U(:,\ j)/\norm{U(:,\ j)}_2$,\Comment{$U\in\mathbb R ^{n\times s}$}
			\For{$h=1,\dots,k$}
			\State $U\gets A\cdot U$
			\For{$j=1,\dots s$}
			\State $H_j(1:h,\ h) =V_j^\top U(:,\ j)$
			\State $U(:,\ j)\gets U(:,\ j)-V_jH_j(1:h,\ h)$
			\State $H_j(h,\ h+1) =  \norm{U(:,\ j)}_2$
			\State $U(:,\ j)\gets U(:,\ j)/H_j(h,\ h+1)$
			\State $V_j\gets[V_j,\ U(:,\ j)]$
			\EndFor
			\EndFor
			\For{$j=1,\dots, s$}
			\State $H_j\gets H_j(1:k,\ :)$
			\State $\mathbf \lambda=$ \texttt{eig($H_j$)}
			\State $\mathbf{ \widetilde \lambda}=$ \texttt{eig($H_j- \norm{u}_2 e_1e_1^\top$)}
			\State $\mathbf {\widehat t}=(\widehat t_j)_{j=1,\dots, k}, \quad t_j=\sum_{h=1}^n (\widetilde \lambda_h^j-\lambda_h^j)$
			\State $T(j,\ :) = \mathbf t + \mathbf{\widehat t}$
			\EndFor
			\State \Return $T$
			\EndProcedure
		\end{algorithmic}
	\end{algorithm} 
\end{minipage}
\subsubsection{A restarted algorithm}

In view of the instability issues related to evaluating \eqref{eq:det-ratio}, we propose to combine Algorithm~\ref{alg:cert-aca2} with a restarting mechanism. Let us assume that the rank of the sought cross approximation is $r$ and that $\bar r< r$ is a small value for which \eqref{eq:det-ratio} can be computed with a sufficient accuracy. We might think at forming the index set $J$ by the incremental application of  Algorithm~\ref{alg:cert-aca2} with input parameter $\bar r$. This means that we first compute  a certified cross approximation of rank $\bar r$ of $A$. Then, we add to the latter a certified cross approximation of rank $\bar r$ of the residual matrix, and so on and so for. The procedure stops when we reach an index set $J$ of cardinality $r$.
We call this method \emph{quasi certified cross approximation} (\textsc{quasi\_cca})  and we report its pseudocode  in Algorithm~\ref{alg:quasi-certified}. The asymptotic cost of \textsc{quasi\_cca} is $r/\bar r$ times the one of \textsc{cca2} for a submatrix of size $\bar r\times \bar r$, that is $\mathcal O(\bar r r n^\omega)$. Even though the cross approximation returned by Algorithm~\ref{alg:quasi-certified} is not guaranteed to verify \eqref{eq:certified}, it is usually the case, as we will see in the numerical results.

\begin{algorithm}[H] 
	\small 
	\caption{}\label{alg:quasi-certified}
	\begin{algorithmic}[1]
		\Procedure{quasi\_cca}{$A,r,\bar r$}
		\State Set $J=\emptyset$
		\While {$r>0$}
		\State $\widehat J=$ \Call{cca2}{$A,\min\{r,\bar r\}$}
		\State $A\gets A-A_{\widehat J}$
		\State $J\gets J\cup\widehat J$
		\State $r=r-\bar r$
		\EndWhile
		\State \Return $J$
		\EndProcedure
	\end{algorithmic}
\end{algorithm}

\subsection{Numerical results}\label{sec:test2}
Let us compare the performances of Algorithm~\ref{alg:aca-cert} and Algorithm~\ref{alg:quasi-certified} on the test matrices $A_1,A_2,A_3,$ $A_5$ introduced in Section\ref{sec:test1}. The bulge chasing procedure used in Algorithm~\ref{alg:aca-cert} has been implemented in Fortran and
is called via a MEX interface. When executing Algorithm~\ref{alg:quasi-certified}, the parameter $\bar r$ has been set to $5$.
\begin{paragraph}{Test 5.} 
	We set  $n=100$, $\rho = 0.85$ and we measure the nuclear norm of the cross approximation error, $\norm{A-A_J}_*$, obtained with \textsc{cca} and \textsc{quasi\_cca} as the parameter $r$ increases. The results are shown in Figure~\ref{fig:cca}, where we also report the upper bound provided by Theorem~\ref{thm:main} and the lower bound $g(r):=\sum_{j\geq r+1}\sigma_j$, corresponding to the approximation error of the truncated SVD (TSVD). We see that, on all examples, the accuracy of \textsc{cca} and \textsc{quasi\_cca} is really close and often the convergence curves are not distinguishable. In addition, in the examples where the decay of the singular values is slow we notice that  Theorem~\ref{thm:main} tends to be pessimistic and the accuracy of \textsc{cca} and \textsc{quasi\_cca} is very close to the one of the TSVD. 
	
	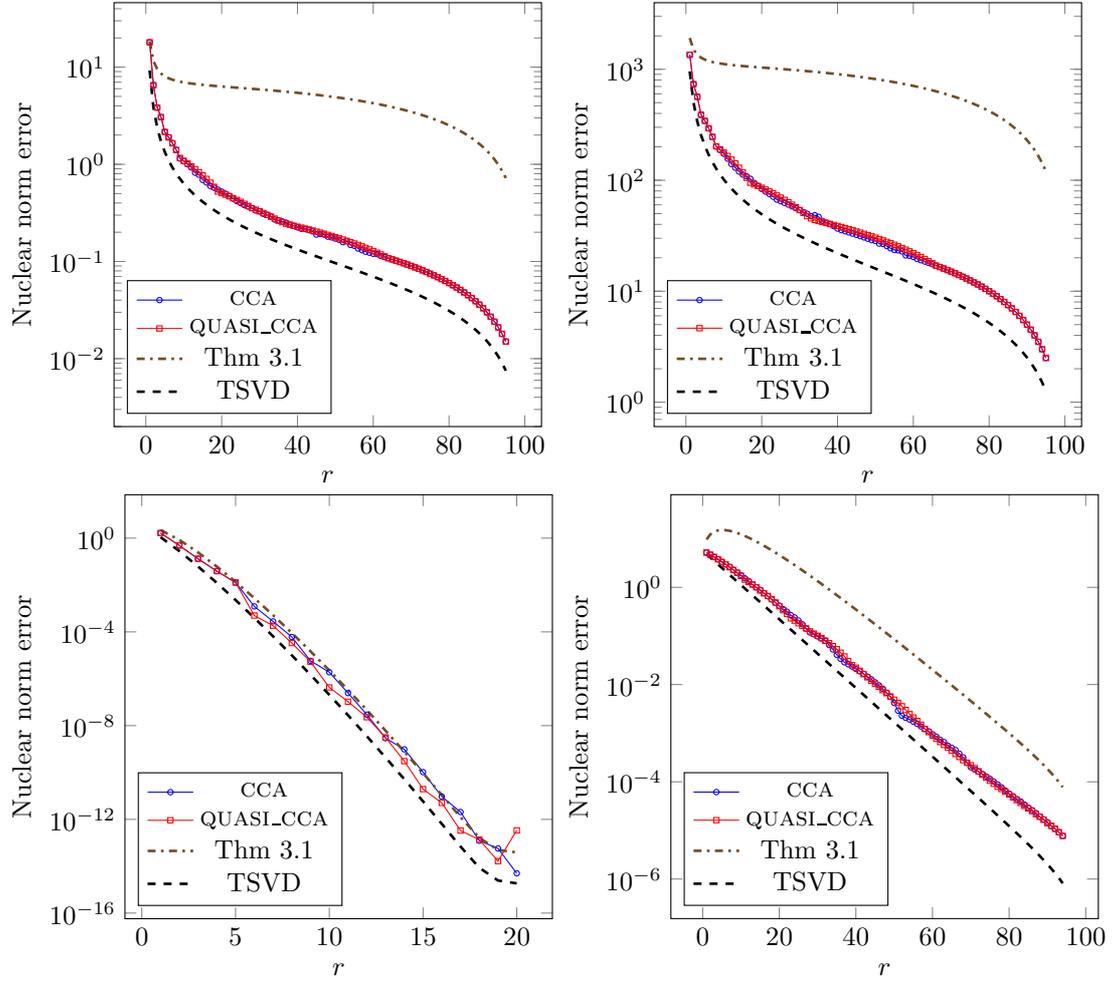
\begin{figure}
		\begin{minipage}{.49\linewidth}
			\begin{tikzpicture}
			\begin{semilogyaxis}[xlabel = $r$, ylabel = Nuclear norm error,width=\linewidth,
			height=1\linewidth,legend pos = south west, ymin =2e-3]
			\addplot+[mark size=1pt, mark = o, blue] table[x index = 0, y index = 2]
			{cca_exp_ker.dat};
			\addplot+[mark size=1pt, mark = square, red] table[x index = 0, y index = 1]
			{cca_exp_ker.dat};		
			\addplot+[line width = 1pt, dashdotted, mark = none] table[x index = 0, y index = 3]
			{cca_exp_ker.dat};	
			\addplot+[line width = 1pt, dashed, mark = none] table[x index = 0, y index = 4]
			{cca_exp_ker.dat};	
			\legend{\textsc{cca}, \textsc{quasi\_cca}, Thm~\ref{thm:main}, TSVD}
			\end{semilogyaxis}
			\end{tikzpicture}
			\begin{tikzpicture}
			\begin{semilogyaxis}[xlabel = $r$, ylabel = Nuclear norm error,width=\linewidth,
			height=1\linewidth,legend pos = south west]
			\addplot+[mark size=1pt, mark = o] table[x index = 0, y index = 2]
			{cca_hilb.dat};
			\addplot+[mark size=1pt, mark = square] table[x index = 0, y index = 1]
			{cca_hilb.dat};		
			\addplot+[line width = 1pt, dashdotted, mark = none] table[x index = 0, y index = 3]
			{cca_hilb.dat};	
			\addplot+[line width = 1pt, dashed, mark = none] table[x index = 0, y index = 4]
			{cca_hilb.dat};	
			\legend{\textsc{cca}, \textsc{quasi\_cca}, Thm~\ref{thm:main}, TSVD}
			\end{semilogyaxis}
			\end{tikzpicture}
		\end{minipage}~\begin{minipage}{.49\linewidth}
			\begin{tikzpicture}
			\begin{semilogyaxis}[xlabel = $r$, ylabel = Nuclear norm error,width=\linewidth,
			height=1\linewidth,legend pos = south west, ymin =0]
			\addplot+[mark size=1pt, mark = o] table[x index = 0, y index = 2]
			{cca_min.dat};
			\addplot+[mark size=1pt, mark = square] table[x index = 0, y index = 1]
			{cca_min.dat};		
			\addplot+[line width = 1pt, dashdotted, mark = none] table[x index = 0, y index = 3]
			{cca_min.dat};	
			\addplot+[line width = 1pt, dashed, mark = none] table[x index = 0, y index = 4]
			{cca_min.dat};	
			\legend{\textsc{cca}, \textsc{quasi\_cca}, Thm~\ref{thm:main}, TSVD}
			\end{semilogyaxis}
			\end{tikzpicture}
			\begin{tikzpicture}
			\begin{semilogyaxis}[xlabel = $r$, ylabel = Nuclear norm error,width=\linewidth,
			height=1\linewidth,legend pos = south west]
			\addplot+[mark size=1pt, mark = o] table[x index = 0, y index = 2]
			{cca_exp.dat};
			\addplot+[mark size=1pt, mark = square] table[x index = 0, y index = 1]
			{cca_exp.dat};		
			\addplot+[line width = 1pt, dashdotted, mark = none] table[x index = 0, y index = 3]
			{cca_exp.dat};	
			\addplot+[line width = 1pt, dashed, mark = none] table[x index = 0, y index = 4]
			{cca_exp.dat};	
			\legend{\textsc{cca}, \textsc{quasi\_cca}, Thm~\ref{thm:main}, TSVD}
			\end{semilogyaxis}
			\end{tikzpicture}
		\end{minipage}
		\caption{Nuclear norm of the error associated with the cross approximations returned by Algorithm~\ref{alg:aca-cert} and Algorithm~\ref{alg:quasi-certified} on the test matrices $A_1$ (top-left), $A_2$ (top-right), $A_3$ (bottom-left) and $A_5$ (bottom-right). All plots report the upper bound provided by Theorem~\ref{thm:main} and the lower bound given by the error of the truncated SVD.}\label{fig:cca}
	\end{figure}
\end{paragraph}
\begin{paragraph}{Test 6.}
	Finally, we test the computational cost of the proposed numerical procedure. We fix $r=20$, $\rho=0.85$ and we run Algorithm~\ref{alg:aca-cert} and Algorithm~\ref{alg:quasi-certified} on $A_5$ for $n\in\{50,100,200,400,800,1600\}$. The timings, reported in Figure~\ref{fig:cca_times}, confirm the cubic complexity with respect to $n$ of Algorithm~\ref{alg:aca-cert}. Although the complexity of the implementation of \textsc{quasi\_cca} is cubic as well (no fast matrix multiplication algorithm has been implemented), it results in a significant gain of computational time due to the more intense use of BLAS 3 operations.  
	
	\begin{figure}
		\centering
		\begin{tikzpicture}
		\begin{loglogaxis}[xlabel = $n$, ylabel = Time (s),width=.5\linewidth,legend pos = south east]
		\addplot+[mark size=1pt, mark = o] table[x index = 0, y index = 2]
		{cca_exp_times.dat};
		\addplot+[mark size=1pt, mark = square] table[x index = 0, y index = 1]
		{cca_exp_times.dat};		
		\addplot[domain=50:1600, dashed]{1e-7*x^3};	 		
		\legend{ \textsc{cca},\textsc{quasi\_cca},$\mathcal O(n^3)$}
		\end{loglogaxis}
		\end{tikzpicture}
		\caption{Timings of Algorithm~\ref{alg:aca-cert} and Algorithm~\ref{alg:quasi-certified} on the test matrix $A_5$ for $r = 20$ and  $n\in\{50, 100,200,$ $400,800,1600\}$.}\label{fig:cca_times}
	\end{figure}
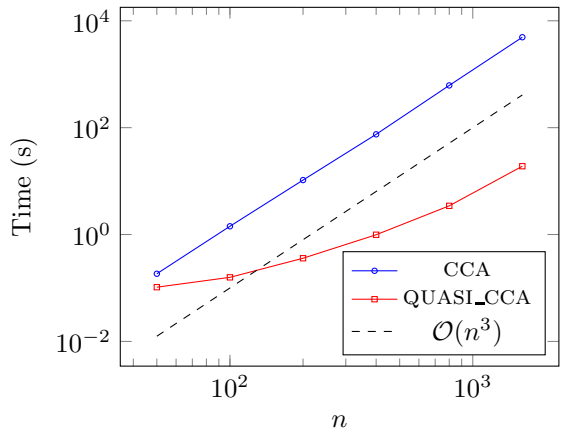
\end{paragraph}
\section{Outlook}
We have proposed several numerical methods for the solution of problems related to the selection of the maximum volume submatrix and the cross approximation of symmetric definite matrices. 

We remark that, the idea used for deriving Algorithm~\ref{alg:aca_ratio} and Algorithm~\ref{alg:mvol_ratio} extends easily to combinatorial optimization problems of the form 
$$
\max_{J\subset \{1,\dots, n\}, \ |J|=r}f(\mathcal V(A_1(J,\ J),\dots, \mathcal V(A_p(J,\ J))
$$
for a multivariate function $f$ and SPSD matrices $A_1,\dots,A_p$.

Also the second part of the manuscript can inspire some future works. For instance, the fact that the maximum volume submatrix of a diagonally dominant matrix is principal might suggest that a result analogous to Theorem~\ref{thm:main} holds also for diagonally dominant matrices. However, it is not straightforward to adjust the proof of Theorem~\ref{thm:main} to this case because we lose the connection between the sum of the volumes of the principal submatrices and the coefficients of the characteristic polynomial. 

Another interesting point is to understand whether the ratio of determinants in \eqref{eq:det-ratio} can be computed with high relative accuracy. This would pave the way to the use of \textsc{cca2} without incorporating any restart mechanisms.

Finally, in the case of large scale matrices one might derive new scalable algorithms for computing cross approximations by combining Algorithm~\ref{alg:aca-cert}--\ref{alg:quasi-certified} with heuristic techniques for reducing the dependence on $n$ in the computational cost. 
\bibliographystyle{abbrv}
\bibliography{library}

\begin{thebibliography}{10}

\bibitem{arioli2015preconditioning}
M.~Arioli and I.~S. Duff.
\newblock Preconditioning linear least-squares problems by identifying a basis
  matrix.
\newblock {\em SIAM Journal on Scientific Computing}, 37(5):S544--S561, 2015.

\bibitem{bebendorf2000approximation}
M.~Bebendorf.
\newblock Approximation of boundary element matrices.
\newblock {\em Numerische Mathematik}, 86(4):565--589, 2000.

\bibitem{bebendorf2008hierarchical}
M.~Bebendorf.
\newblock {\em Hierarchical {M}atrices}.
\newblock Springer, 2008.

\bibitem{beckermann2018low}
B.~Beckermann, D.~Kressner, and M.~Schweitzer.
\newblock Low-rank updates of matrix functions.
\newblock {\em {SIAM} Journal on Matrix Analysis and Applications},
  39(1):539--565, 2018.

\bibitem{ccivril2009selecting}
A.~{\c{C}}ivril and M.~Magdon-Ismail.
\newblock On selecting a maximum volume sub-matrix of a matrix and related
  problems.
\newblock {\em Theoretical Computer Science}, 410(47-49):4801--4811, 2009.

\bibitem{cortinovis2019low}
A.~Cortinovis and D.~Kressner.
\newblock Low-rank approximation in the {F}robenius norm by column and row
  subset selection.
\newblock {\em arXiv preprint arXiv:1908.06059}, 2019.

\bibitem{cortinovis2020maximum}
A.~Cortinovis, D.~Kressner, and S.~Massei.
\newblock On maximum volume submatrices and cross approximation for symmetric
  semidefinite and diagonally dominant matrices.
\newblock {\em Linear Algebra and its Applications}, 593:251--268, 2020.

\bibitem{cuppen1980divide}
J.~J. Cuppen.
\newblock A divide and conquer method for the symmetric tridiagonal
  eigenproblem.
\newblock {\em Numerische Mathematik}, 36(2):177--195, 1980.

\bibitem{demko1984decay}
S.~Demko, W.~F. Moss, and P.~W. Smith.
\newblock Decay rates for inverses of band matrices.
\newblock {\em Mathematics of Computation}, 43(168):491--499, 1984.

\bibitem{demmel1999computing}
J.~Demmel, M.~Gu, S.~Eisenstat, I.~Slapni{\v{c}}ar, K.~Veseli{\'c}, and
  Z.~Drma{\v{c}}.
\newblock Computing the singular value decomposition with high relative
  accuracy.
\newblock {\em Linear Algebra and its Applications}, 299(1-3):21--80, 1999.

\bibitem{demmel2004accurate}
J.~Demmel and P.~Koev.
\newblock Accurate {SVD}s of weakly diagonally dominant {M}-matrices.
\newblock {\em Numerische Mathematik}, 98(1):99--104, 2004.

\bibitem{deshpande2010efficient}
A.~Deshpande and L.~Rademacher.
\newblock Efficient volume sampling for row/column subset selection.
\newblock In {\em 2010 IEEE 51st annual symposium on foundations of computer
  science}, pages 329--338. IEEE, 2010.

\bibitem{fasshauer2015kernel}
G.~E. Fasshauer and M.~J. McCourt.
\newblock {\em Kernel-based approximation methods using Matlab}, volume~19.
\newblock World Scientific Publishing Company, 2015.

\bibitem{golub1973some}
G.~H. Golub.
\newblock Some modified matrix eigenvalue problems.
\newblock {\em SIAM Review}, 15(2):318--334, 1973.

\bibitem{goreinov2010find}
S.~A. Goreinov, I.~V. Oseledets, D.~V. Savostyanov, E.~E. Tyrtyshnikov, and
  N.~L. Zamarashkin.
\newblock How to find a good submatrix.
\newblock In {\em Matrix Methods: Theory, Algorithms And Applications:
  Dedicated to the Memory of Gene Golub}, pages 247--256. World Scientific,
  2010.

\bibitem{goreinov2001maximal}
S.~A. Goreinov and E.~E. Tyrtyshnikov.
\newblock The maximal-volume concept in approximation by low-rank matrices.
\newblock {\em Contemporary Mathematics}, 280:47--52, 2001.

\bibitem{gu1996efficient}
M.~Gu and S.~C. Eisenstat.
\newblock Efficient algorithms for computing a strong rank-revealing {QR}
  factorization.
\newblock {\em SIAM Journal on Scientific Computing}, 17(4):848--869, 1996.

\bibitem{gu2004strong}
M.~Gu and L.~Miranian.
\newblock Strong rank revealing cholesky factorization.
\newblock {\em Electronic Transactions on Numerical Analysis}, 17:76--92, 2004.

\bibitem{haber2016sparse}
A.~Haber and M.~Verhaegen.
\newblock Sparse solution of the {L}yapunov equation for large-scale
  interconnected systems.
\newblock {\em Automatica}, 73:256--268, 2016.

\bibitem{harbrecht2012low}
H.~Harbrecht, M.~Peters, and R.~Schneider.
\newblock On the low-rank approximation by the pivoted {C}holesky
  decomposition.
\newblock {\em Applied Numerical Mathematics}, 62(4):428--440, 2012.

\bibitem{higham2002accuracy}
N.~J. Higham.
\newblock {\em Accuracy and stability of numerical algorithms}, volume~80.
\newblock SIAM, 2002.

\bibitem{kiefer1961optimum}
J.~Kiefer.
\newblock Optimum experimental designs {V}, with applications to systematic and
  rotatable designs.
\newblock In {\em Proceedings of the fourth Berkeley symposium on mathematical
  statistics and probability}, volume~1, pages 381--405. Univ of California
  Press, 1961.

\bibitem{kressner2020certified}
D.~Kressner, J.~Latz, S.~Massei, and E.~Ullmann.
\newblock Certified and fast computations with shallow covariance kernels.
\newblock {\em arXiv preprint arXiv:2001.09187}, 2020.

\bibitem{mikhalev2018rectangular}
A.~Mikhalev and I.~V. Oseledets.
\newblock Rectangular maximum-volume submatrices and their applications.
\newblock {\em Linear Algebra and its Applications}, 538:187--211, 2018.

\bibitem{oseledets2010tt}
I.~Oseledets and E.~Tyrtyshnikov.
\newblock {TT}-cross approximation for multidimensional arrays.
\newblock {\em Linear Algebra and its Applications}, 432(1):70--88, 2010.

\bibitem{osinsky2018rectangular}
A.~Osinsky.
\newblock Rectangular maximum volume and projective volume search algorithms.
\newblock {\em arXiv preprint arXiv:1809.02334}, 2018.

\bibitem{reedmethods}
M.~Reed and B.~Simon.
\newblock Methods of modern mathematical physics, vol. 4 analysis of operators,
  1978.

\bibitem{rehman2011computing}
R.~Rehman and I.~C. Ipsen.
\newblock Computing characteristic polynomials from eigenvalues.
\newblock {\em SIAM Journal on Matrix Analysis and Applications},
  32(1):90--114, 2011.

\bibitem{sommariva2009computing}
A.~Sommariva and M.~Vianello.
\newblock Computing approximate {F}ekete points by {QR} factorizations of
  vandermonde matrices.
\newblock {\em Computers \& Mathematics with Applications}, 57(8):1324--1336,
  2009.

\bibitem{stewart1998matrix}
G.~W. Stewart.
\newblock {\em Matrix Algorithms: Volume 1: Basic Decompositions}.
\newblock SIAM, 1998.

\bibitem{summa2014largest}
M.~D. Summa, F.~Eisenbrand, Y.~Faenza, and C.~Moldenhauer.
\newblock On largest volume simplices and sub-determinants.
\newblock In {\em Proceedings of the twenty-sixth annual ACM-SIAM symposium on
  Discrete algorithms}, pages 315--323. SIAM, 2014.

\bibitem{tyrtyshnikov2000incomplete}
E.~Tyrtyshnikov.
\newblock Incomplete cross approximation in the mosaic-skeleton method.
\newblock {\em Computing}, 64(4):367--380, 2000.

\bibitem{zamarashkin2018existence}
N.~Zamarashkin and A.~Osinsky.
\newblock On the existence of a nearly optimal skeleton approximation of a
  matrix in the {F}robenius norm.
\newblock In {\em Doklady Mathematics}, volume~97, pages 164--166. Springer,
  2018.

\end{thebibliography}
	\end{document}